\documentclass[11pt,reqno,fleqn,preprint,authoryear]{elsarticle}
\usepackage{amsmath, amsthm, amssymb}
\usepackage{verbatim,enumerate,graphicx,bbm}
\usepackage{hyperref,environ,enumitem}
\usepackage[authoryear]{natbib}
\usepackage{algorithm2e}
\allowdisplaybreaks[2]

\def\R{\mathbb{R}}

\numberwithin{equation}{section}
\newtheorem{thm}{Theorem}

\newtheorem{prop}{Proposition}
\theoremstyle{remark}

\theoremstyle{definition}

\newcommand{\E}[1]{\operatorname{E}\!\left[#1\right]}

\newcommand{\tr}[1]{\operatorname{tr}\!\left(#1\right)}
\newcommand{\Tr}{\text{tr}}

\newcommand{\rk}[1]{\operatorname{rk}#1}

 \newcommand{\diag}{\mathop{\rm diag}}

\usepackage{tikz}

\RestyleAlgo{boxruled}

\usepackage{caption}
\usepackage{subcaption}

\begin{document}

\title{Improved Second Order Estimation in the
	Singular Multivariate Normal Model}

\author[hec]{Didier Ch\'etelat\fnref{thanks}}
\address[hec]{Department of Decision Sciences, HEC Montr\'eal}
\ead{didier.chetelat@hec.ca}

\author[cornell]{Martin T. Wells\fnref{thanks}\corref{main}}
\address[cornell]{Department of Statistical Science, Cornell University}
\ead{mtw1@cornell.edu}

\cortext[main]{Corresponding author}

\fntext[thanks]{This research was partially supported by NSF grants DMS-1208488 and CCF-0808864 and NIH U19 AI111143.}

\begin{abstract}
We consider the problem of estimating covariance and precision matrices, and their associated discriminant coefficients, from normal data when the rank of the covariance matrix is strictly smaller than its dimension and the available sample size. Using unbiased risk estimation, we construct novel estimators by minimizing upper bounds on the difference in risk over several classes. Our proposal estimates are empirically demonstrated to offer substantial improvement over classical approaches.
\end{abstract}


\maketitle


\section{Introduction}

With the recent explosion of high throughput data, much interest has arisen in applications where the number of feature parameters is greater than the sample size. In this situation, it is typically assumed that, despite their number, the underlying components are linearly independent, or in other words that their covariance matrix has full rank. However, little attention has been given to the situation where there is dependence between the components, that is, where the covariance matrix would be singular.

Recently, \cite{TsukumaKubokawa14} investigated the problem of estimating the mean vector of a multivariate normal distribution when the unknown covariance matrix is singular. By deriving an unbiased risk estimator for the quadratic loss, they were able to give sufficient conditions for an estimator to dominate the maximum likelihood estimator.

This article is concerned with the same model as \cite{TsukumaKubokawa14}, but we consider three different estimation problems. Unlike the mean estimation problem, all three estimation scenarios depend the second order moment of the distribution. In each case we provide decision-theoretic results that lead to improved inference. The first task is the estimation of the singular covariance matrix itself, under an invariant squared loss. This problem was first considered in the full rank case by \cite{Haff80}, and in the high-dimensional setting by \cite{Konno09}. The second concern is the estimation of the Moore-Penrose pseudo-inverse of the covariance matrix, also known as the precision matrix, under the Frobenius loss. This problem was first considered in the full rank case by \cite{Haff77,Haff79b} and in the high-dimensional setting by \cite{KubokawaSrivastava08}.

Finally, we consider the problem of estimating the discriminant coefficient that arise in Linear Discriminant Analysis (LDA) under the squared loss, a problem first considered in the full rank case by \cite{Haff86} and \cite{DeySrinivasan91}. LDA is a standard method for classification when the number
of observations $n$ is much larger than the number of features $p$. If data follows $p$-variate
normal distribution with the same covariance structure across the groups, it provides an
asymptotically optimal classification rule, meaning that its misclassification error converges
to Bayes risk. However, it was noted by \cite{dudoit2002comparison} that a naive
implementation of LDA for high-dimensional data provides poor classification results in
comparison to alternative methods. A rigorous proof of this phenomenon in the case $p\gg n$
is given by \cite{bickel2004}. There are two main reasons for this. First, standard LDA
uses the sample covariance matrix to estimate the covariance structure, and in high dimensional
settings this results in a singular estimator. Secondly, by using all $p$ features in classification,
interpretation of the results becomes challenging.  One of the popular approaches to deal with the singularity is to use the independence rule which overcomes the singularity problem of the sample covariance but ignores the dependency structure. This approach is very appealing because of its simplicity and was encouraged by the work of \cite{bickel2004}, who showed it performs better than the standard LDA in a $p \gg n$ setting when the population matrix is full rank. Unfortunately, independence is only an approximation and it is unrealistic in most applications: for instance, in a genomic context, gene interactions and low dimensional network structure are crucial for the understanding of biological processes. In this situation, one should aim for better estimators of the covariance matrix rather than relying on an independence structure that assumes a full rank population covariance matrix.  Indeed, we will see in Section \ref{sec:simulations} that using the diagonal of the sample covariance matrix is a poor strategy if the true covariance matrix is rank deficient.

The presentation of our approach to these three estimation problems is divided as follows. The decision-theoretic results are described in Section \ref{sec:results}. For each of the three problems, we construct an appropriate unbiased estimator of the risk (URE) using Stein's and Haff's lemmas \citep{Stein86,Haff79a, TsukumaKubokawa14}. We then consider the class of estimator given by constant multiples of a naive estimator, and minimize an upper bound on the difference in risk to obtain estimators that dominate the naive estimator. Finally, we consider a larger class given by the sum of this estimator and an appropriate trace, and again minimize an upper bound on the risk to obtain a dominating estimator.

In Section \ref{sec:simulations}, we investigate the amount of improvement provided by the proposed estimators through numerical study. Finally, proofs of the statements of Section \ref{sec:results} are provided in Section \ref{sec:proofs}.

\section{Estimation Results}\label{sec:results}

\subsection{Model}\label{subsec:model}

Our setting is similar to the one used in \cite{TsukumaKubokawa14}. We observe an $n$-sample $X_1,...,X_n$ identically and independently distributed from a $p$-dimensional multivariate normal distribution $N_p(\mu,\Sigma)$, where $\mu$ and $\Sigma$ are unknown. However, the $p$-dimensional covariance matrix $\Sigma$ is rank-deficient with respect to the dimension and the sample size, in the sense that
\begin{align}
r=\rk{\Sigma}<\min(n,p).
\label{eq:singular}
\end{align}
The resulting singular multivariate normal distribution does not have a density with respect to the Lebesgue measure on $\R^p$, but lives in the $r$-dimensional linear subspace spanned by the columns of $\Sigma$. More details can be found, for example, in \citet[Section 2.1]{SrivastavaKhatri79}.

Define the $n\times p$ data matrix $X=(X_1,...,X_p)^t$. The sample covariance matrix $S=(X-1_n\bar X^t)^t(X-1_n\bar X^t)/n$ then follows a Wishart distribution $W_p(n-1,\Sigma/n)$ with $n-1$ degrees of freedom. Since $\Sigma$ is rank-deficient, it is singular in the terminology of \citet[Section 3.1]{SrivastavaKhatri79}. We warn the reader that the expression ``singular Wishart'' has also been used in the literature to describe the different situation where the covariance is positive-definite and the dimension exceeds the degrees of freedom, as in \citet{Srivastava03}. Let $S=O_1LO_1^t$ denote the reduced spectral decomposition of $S$, where $L=\text{diag}(l_1,...,l_r)$ denote the $r$ non-zero eigenvalues and $O_1$ is $p\times r$ semi-orthogonal.

In this situation, neither $S$ nor $\Sigma$ are invertible. Since inverses of covariance matrix are of considerable interest in multivariate statistical analysis, some generalized inverse of these quantities is desirable. In this article, we will focus on the Moore-Penrose pseudoinverse, which will be denoted $A^+$ for a matrix $A$. Definitions and theoretical properties can be found in \citet[Chapter 20]{Harville97}.

The singular multivariate normal model is amenable to decision-theoretic analysis through a key insight of \citet[Section 2.2]{TsukumaKubokawa14}. The authors proved that when \eqref{eq:singular} holds, the subspace spanned by the sample covariance matrix is almost surely constant and matches the subspace spanned the true covariance matrix, in the sense that the remarkable identity holds
\begin{align}
SS^+=\Sigma\Sigma^+.
\label{eq:SSSS}
\end{align}
This fact will be repeatedly used in the Section \ref{sec:proofs} proofs and is essential to our derivations.

Let us now turn our attention to the three problems we wish to solve. In terms of the notation introduced above, these are:
\vspace{5pt}
\begin{itemize}
\setlength\itemsep{5pt}
\setlength{\itemindent}{-30pt}
\item[] \textit{Covariance matrix estimation.} The estimation of $\Sigma$ under the invariant squared loss $L(\hat\Sigma,\Sigma)=\Tr[(\hat{\Sigma}\Sigma^+-I_p)^2]$.
\item[] \textit{Precision matrix estimation.} The estimation of $\Sigma^+$ under the Frobenius loss $L(\hat\Sigma^+,\Sigma^+)=\|\hat\Sigma^+-\Sigma^+\|_F^2$.
\item[] \textit{Discriminant coefficient estimation.} The estimation of $\eta=\Sigma^+\mu$ under the square loss $L(\hat\eta,\eta)=\|\hat\eta-\eta\|_2^2$.
\end{itemize}
\vspace{5pt}
The traditional estimators for $\mu$ and $\Sigma$ are the sample mean and covariance $(\bar X, S)$, which suggests the corresponding naive estimators $S$, $S^+$ and $S^+\bar X$ for each respective problem. In the next three subsections we will see traditional estimators are not admissible and improved estimators will be developed.

\subsection{Covariance matrix estimation}\label{subsec:cov}

The standard estimator for a covariance matrix is the sample covariance matrix $S$. An alternative is the unbiased estimator $\frac{n}{n-1}S$, which corrects for the loss in degrees of freedom from not knowing $\mu$. We will look for estimators that improve over these benchmarks and study their performance.

We first show that an unbiased estimator of the risk holds for orthogonally invariant estimators, that is, estimators of the form $\hat\Sigma=O_1\Psi O_1^t$ with $\Psi=\text{diag}(\psi_1,...,\psi_r)$ twice-differentiable functions of $L=\text{diag}(l_1,...,l_r)$.

\begin{thm}[Unbiased risk estimation for singular covariance matrices]\label{thm:cov-ure} Let $1\leq r\leq n-1$ and define \[\psi^*_k=\left[\frac{n-r-2}n\frac{\psi_k}{l_k}+\frac4n\frac{\partial \psi_k}{\partial l_k}+\frac2n\sum_{b\neq k}^r\frac{\psi_k-\psi_b}{l_k-l_b}-2\right]\psi_k.\]
Assume the regularity conditions
{\setlength{\mathindent}{5pt}\begin{align}&
\E{\left|p
+\sum_{k=1}^r\frac{n-r-2}n\frac{\psi_k^*}{l_k}
+\frac2n\sum_{k=1}^r\frac{\partial \psi_k^*}{\partial l_k}
+\frac1n\sum_{k\neq b}^r\frac{\psi_k^*-\psi_b^*}{l_k-l_b}
\right|}<\infty,
\notag\\&
\E{\left|p+\sum_{k=1}^r \frac{n-r-2}n\frac{\psi_k}{l_k}
+\frac2n\sum_{k=1}^r\frac{\partial \psi_k}{\partial l_k}
+\frac1n\sum_{k\neq b}^r\frac{\psi_k-\psi_b}{l_k-l_b}\right|}<\infty,
\notag\\&\hspace{20pt}
\E{\sum_{k=1}^r\left|\frac{\psi^*_k}{l_k}\right|^2}<\infty
\text{ and }
\E{\sum_{k=1}^r\left|\frac{\psi_k}{l_k}\right|^2}<\infty.
\label{eq:cov-ure-conditions}
\end{align}}
We then have
{\setlength{\mathindent}{5pt}\begin{align}&
\E{\Tr\Big(\big[\hat\Sigma\Sigma^+-I_p\big]^2\Big)}
\notag\\&\qquad
=\E{p
+\frac{n-r-2}{n}\sum_{k=1}^r\frac{\psi^*_k}{l_k}
+\frac{2}{n}\sum_{k=1}^r\frac{\partial\psi^*_k}{\partial l_k}
+\frac1{n}\sum_{k\not= b}^r\frac{\psi^*_k-\psi^*_b}{l_k-l_b}
}.
\label{eq:cov-ure-ure}
\end{align}}
\end{thm}

Let us now consider estimators that are proportional to the sample covariance matrix, that is, of the form $aS$ for $a$ constant. The following result provides the optimal proportionality factor.

\begin{prop}\label{prop:cov-aS} Let $1\leq r\leq n-1$.  The optimal estimator of $\Sigma$ of the form $aS$ for $a\in\R$ a deterministic constant is
$\hat{\Sigma}_\text{HF1}=\frac{n}{n+r}S$, with risk
\[\E{\Tr\Big(\big[\hat\Sigma_\text{HF1}\Sigma^+-I_p\big]^2\Big)}
=p-\frac{(n-1)r}{n+r}.\]
In particular $\hat{\Sigma}_\text{HF1}$ dominates $S$, which itself dominates $\frac{n}{n-1}S$.
\end{prop}

Thus $\frac{n}{n-1}S$ and $S$ are inadmissible. We can further extend this result by considering a larger class of estimators of the form $\frac{n}{n+r}\left[S+tSS^+\,\Tr^{-1}(S^+)\right]$ for $t$ constant. Estimators of this shape were first considered by \citet{Haff80}. Although computing the exact risk of these estimators is difficult, it is possible to bound the difference in risk with the one of $\hat\Sigma_\text{HF1}$ as follows.

\begin{prop}\label{prop:cov-aS+tTr} Let $1\leq r\leq n-4$. Then the risk of estimators of the form $\hat\Sigma_t=\frac{n}{n+r}\left[S+tSS^+\,\Tr^{-1}(S^+)\right]$ for $t\in\R$ can be bounded by
{\setlength{\mathindent}{5pt}\begin{align}&
\E{\Tr\Big(\big[\hat\Sigma_t\Sigma^+-I_p\big]^2\Big)}
\leq\;
\E{\Tr\Big(\big[\hat\Sigma_\text{HF1}\Sigma^+-I_p\big]^2\Big)}
\notag\\&\hspace{40pt}
+\bigg[
\frac{(n-r)(n-r+2)}{(n+r)^2}t^2
-2\frac{(n-r)(r-1)}{(n+r)^2}t
\bigg]\E{\frac{\Tr(S^{+2})}{\text{tr}^2(S^+)}}.
\label{eq:cov-aS+tTr-risk}
\end{align}}
The constant that minimizes this upper bound is $t=\frac{r-1}{n-r+2}$. When $r>1$, the estimator $\hat{\Sigma}_\text{HF2} =\frac{n}{n+r}\left[S+\frac{r-1}{n-r+2}SS^+\Tr^{-1}(S^+)\right]$ dominates $\hat{\Sigma}_\text{HF1}$.
\end{prop}

Thus $\hat\Sigma_\text{HF1}$ is itself inadmissible for $r>1$. Although this result does not show $\hat\Sigma_\text{HF2}$ optimal within the class, the estimator is likely to have good overall risk properties.

\subsection{Precision matrix estimation}\label{subsec:pres}

A standard estimator for a singular precision matrix is the Moore-Penrose pseudoinverse of the sample covariance matrix $S^+$. Note that by \citet[Page 97, Equation (12)]{Muirhead82} we have
\[\E{S^+}=\frac{n}{n-r-2}\Sigma^+.\]
for $n-r-2>0$. Thus in this case an alternative could be the unbiased estimator $\frac{n-r-2}{n}S^+$. We will look for estimators that improve over these benchmarks and study their performance.

We first show that an unbiased estimator of the risk holds for orthogonally invariant estimators, that is, estimators of the form $\hat\Sigma^+=O_1\Psi O_1^t$ with $\Psi=\text{diag}(\psi_1,...,\psi_r)$ twice-differentiable functions of $L=\text{diag}(l_1,...,l_r)$.

\begin{thm}[Unbiased risk estimation for singular precision matrices]\label{thm:pres-ure} Let $1\leq r\leq n-1$.
Assume the regularity condition
{\setlength{\mathindent}{5pt}\begin{align*}&
\quad
\E{\bigg|
\frac{n-r-2}{n}\sum_{k=1}^p\frac{\psi_k}{l_k}
+\frac2n\sum_{k=1}^r\frac{\partial\psi_k}{\partial l_k}
+\frac1n\sum_{k\not=b}\frac{\psi_k-\psi_b}{l_k-l_b}\bigg|}<\infty.
\end{align*}}
Then
{\setlength{\mathindent}{0pt}\begin{align*}&
\E{\|\hat{\Sigma}^+-\Sigma^+\|_F^2}
\\&
=\E{
\sum_{k=1}^r\psi_k^2
-2\frac{n-r-2}n\sum_{k=1}^r\frac{\psi_k}{l_k}
-\frac4n\sum_{k=1}^r\frac{\partial\psi_k}{\partial l_k}
-\frac2n\sum_{k\not=b}^r\frac{\psi_k-\psi_b}{l_k-l_b}
}\hspace{-5pt}
+\tr{\Sigma^{-2}}\!.
\end{align*}}
\end{thm}

Let us now consider estimators that are proportional to the Moore-Penrose inverse of the sample covariance matrix, that is, of the form $aS^+$ for $a$ constant. The following optimality result holds over this class.

\begin{prop}\label{prop:pres-aS} Let $1\leq r\leq n-5$. The risk of estimators of the form $aS^+$ for $a\leq\frac{n-r-2}n$ can be bounded in terms of the risk of $\frac{n-r-2}{n}S^+$ by
{\setlength{\mathindent}{5pt}\begin{align}&
\E{\|aS^+-\Sigma^+\|_F^2}
\leq\;
\E{\Big\|\frac{n-r-2}nS^+-\Sigma^+\Big\|_F^2}
\notag\\&\hspace{120pt}
+\bigg(a-\frac{n-r-2}{n}\bigg)
\bigg(a-\frac{n-r-6}{n}\bigg)\E{\Tr(S^{+2})}.
\label{eq:pres-aS-risk}
\end{align}}
The constant that minimizes this upper bound is $a=\frac{n-r-4}{n}$, and the corresponding estimator $\hat{\Sigma}^+_\text{EM1} =\frac{n-r-4}{n}S^+$ dominates $\frac{n-r-2}{n}S^+$, which itself dominates $S^+$.
\end{prop}

Thus $\frac{n-r-2}{n}S^+$ and $S^+$ are inadmissible. Note that our bound on the risk only holds for $a\leq\frac{n-r-2}n$: presumably, estimators $aS^+$ with $a>\frac{n-r-2}n$ do not dominate $\frac{n-r-2}nS^+$, but we have not been able to prove this hypothesis.

In any case, we can further extend this result by considering a larger class of estimators of the form $\frac{n-r-4}n\left[S^+ +t\,SS^+\Tr^{-1}(S)\right]$ for $t$ constant. Estimators of this form were first considered by \cite{EfronMorris76}. It is possible to bound the difference in risk with the one of $\hat\Sigma^+_\text{EM1}$ as follows.

\begin{prop}\label{prop:pres-aS-tTr} Let $1\leq r\leq n-5$. The risk of estimators of the form $\hat\Sigma^+_t=\frac{n-r-4}n\left[S^++t\,SS^+\Tr^{-1}(S)\right]$ for $t\in\R$ can be bounded in terms of the risk of $\hat\Sigma^+_\text{EM1}=\frac{n-r-4}{n}S^+$ through
{\setlength{\mathindent}{5pt}\begin{align}&
\E{\|\hat\Sigma^+_t-\Sigma^+\|_F^2}
\leq\;
\E{\Big\|\hat\Sigma^+_\text{EM1}-\Sigma^+\Big\|_F^2}
\notag\\&\hspace{60pt}
+\frac{(n-r-4)r}{n^2}\bigg[
(n-r-4)t^2
-4(r-1)t
\bigg]
\E{\frac1{\Tr^{2}(S)}}.
\label{eq:pres-aS-tTr-risk}
\end{align}}
The constant that minimizes this upper bound is $t=2\frac{r-1}{n-r-4}$, and the corresponding estimator $\hat\Sigma^+_\text{EM2} =\frac{n-r-4}n\left[S^++2\frac{r-1}{n-r-4}SS^+\Tr^{-1}(S)\right]$ dominates $\hat\Sigma^+_\text{EM1}$.
\end{prop}

Thus $\hat\Sigma^+_\text{EM1}$ is itself inadmissible. Again, although these results do not show $\hat\Sigma^+_\text{EM1}$ and $\hat\Sigma^+_\text{EM2}$ optimal within their classes, they are likely to possess good overall risk properties.

\subsection{Discriminant coefficients estimation}\label{subsec:disc}

A standard estimator for a singular discriminant coefficient is $S^+\bar X$. Note that since $\bar X$ and $S$ are independent, we have
{\setlength{\mathindent}{10pt}\begin{align*}
\E{S^+\bar X}=\frac{n}{n-r-2}\Sigma^+\mu
\end{align*}}
for $n-r-2>0$. Thus in this case an alternative could be the unbiased estimator $\frac{n-r-2}{n}S^+\bar X$. We will look for estimators that improve over these benchmarks and study their performance.

We first show that an unbiased estimator of the risk holds for  estimators of the form $\hat\eta=O_1\Psi O_1^t\bar X$ with $\Psi=\text{diag}(\psi_1,...,\psi_r)$ twice-differentiable functions of $L=\text{diag}(l_1,...,l_r)$.

\begin{thm}[Unbiased risk estimation for singular discriminant coefficients]\label{thm:disc-ure}
Let $\Psi^*=\text{diag}(\psi^*_1,...,\psi^*_r)$ with
\begin{align*}
\psi^*_k=\frac{n-r-2}n\frac{\psi_k}{l_k}+\frac2n\frac{\partial\psi_k}{\partial l_k}+\frac1n\sum_{b\not=k}^r\frac{\psi_k-\psi_b}{l_k-l_b}.
\end{align*}
Assume the regularity conditions
\begin{align*}
\E{\left|\sum_{k=1}^r\psi_k\right|}<\infty
\text{ and }
\E{\sum_{k=1}^r\Big|\psi^*_k\Big|}<\infty.
\end{align*}
Then
\begin{align*}
\E{\Big\|\hat\eta-\eta\Big\|_2^2}
&=\E{\frac2n\text{tr}\,\hat{\Sigma}^++\bar X^tO_1(\Psi^2-2\Psi^*)O_1^t\bar X}
\\&\qquad-\E{\vphantom{\bigg\vert}(\bar X-\mu)^t\Sigma^{+2}(\bar X+\mu)}.
\end{align*}
\end{thm}

Let us now consider estimators that are proportional to the naive estimator, that is, of the form $aS^+\bar X$ for $a$ constant. The following optimality result holds over this class.

\begin{prop}\label{prop:disc-aS} Let $1\leq r\leq n-5$. The risk of estimators of the form $aS^+\bar X$ for $a\leq\frac{n-r-2}n$ can be bounded in terms of the risk of $\frac{n-r-2}{n}S^+\bar X$ by
{\setlength{\mathindent}{5pt}\begin{align}&
\E{\Big\|aS^+\bar X-\eta\Big\|_2^2}
\leq\;
\E{\Big\|\frac{n-r-2}nS^+\bar X-\eta\Big\|_2^2}
\notag\\&\hspace{60pt}
+\bigg(a-\frac{n-r-2}{n}\bigg)\bigg(a-\frac{n-r-4}{n}\bigg)\,\text{E}\Big(\bar X^tS^{+2}\bar X\Big).
\label{eq:disc-aS-risk}
\end{align}}
The constant that minimizes this upper bound is $a=\frac{n-r-3}n$, and the corresponding estimator $\hat{\eta}_\text{TK1}=\frac{n-r-3}nS^+\bar X$ dominates $\frac{n-r-2}{n}S^+\bar X$, which itself dominates $S^+\bar X$.
\end{prop}

Thus $\frac{n-r-2}{n}S^+$ and $S^+$ are inadmissible. Again, note that our bound on the risk only holds on the subset $a\leq\frac{n-r-2}n$. Presumably, estimators $aS^+$ with $a>\frac{n-r-2}n$ do not dominate $\frac{n-r-2}nS^+\bar X$, but we have not been able to prove this result.

We can further extend this result by considering a larger class of estimators of the form $\frac{n-r-3}n\left[S^++t\,SS^+\Tr^{-1}(S)\right]\bar X$ for $t$ constant. Estimators of this form were first considered by \cite{DeySrinivasan91}. It is possible to bound the difference in risk with the one of $\hat\eta_\text{TK1}$ as follows.

\begin{prop}\label{prop:disc-aS-tTr} Let $1\leq r\leq n-5$. The risk of estimators of the form $\hat\eta_t=\frac{n-r-3}n\left[S^++t\,SS^+\Tr^{-1}(S)\right]\bar X$ for $t\in\R$ can be bounded in terms of the risk of $\eta_\text{TK1}=\frac{n-r-3}{n}S^+\bar X$ through
{\setlength{\mathindent}{5pt}\begin{align}&
\E{\|\hat\eta_t-\eta\|_2^2}
\leq\;
\E{\Big\|\hat\eta_\text{TK1}-\eta\|_2^2}
\notag\\&\hspace{60pt}
+\frac{(n-r-3)}{n^2}\bigg[2(r+1)t
+(n-r-3)t^2\bigg]\E{\frac1{\Tr(S)}}.
\label{eq:disc-aS-tTr-risk}
\end{align}}
The constant that minimizes this upper bound is $t=-\frac{r+1}{n-r-3}$, and the corresponding estimator $\hat\eta_\text{TK2} =\frac{n-r-3}n\left[S^+-\frac{r+1}{n-r-3}SS^+\Tr^{-1}(S)\right]\bar X$ dominates $\hat\eta_\text{TK1}$.
\end{prop}

Thus $\hat\eta_\text{TK1}$ is itself inadmissible. One again, although these results do not show $\hat\eta_\text{TK1}$ and $\hat\eta_\text{TK2}$ optimal within their classes, they are likely to have good overall risk properties.

\section{Numerical study}\label{sec:simulations}

In this section we investigate the risk performance of the proposed estimator for covariance, precision and discriminant coefficients estimation through two simulation studies.  We also consider the performance of the diagonal of the sample covariance matrix, $\text{diag}(S)$. In various applications, using this estimator is a popular approach to overcome the singularity problem of the sample covariance. Although it ignores the dependency structure, this estimator is appealing because of its simplicity, and was suggested by the results of \cite{bickel2004}.

\subsection{Autoregressive simulation}\label{subsec:AR}

We let $(n,p)$ be $(150,100)$, $(200,100)$, $(200,150)$ and $(250,150)$. For each $r$ from $1$ to $(n-4)\wedge p$, we constructed the true covariance matrix $\Sigma$ from an autoregressive structure with coefficient $0.9$ and set its $p-r$ smallest eigenvalues to zero to create a rank $r$ matrix, as described in Algorithm \ref{alg:sigma}. We then randomly generated $1,000$ replications from a multivariate normal distribution with mean $\mu=(1,\dots,1)$ and singularized autoregressive covariance $\Sigma$, and computed the resulting sample covariance matrix $S=X^tX/n$.

\begin{algorithm}[ht]
\KwData{$p$, $r$}
\KwResult{$\Sigma$}
\For{$i,j\in\{1,...,p\}$}{
$\Sigma_{ij}=0.5^{|i-j|}$
}
\For{$k\in\{r+1,...,p\}$}{
$\lambda_k(\Sigma)=0$
}
\caption{Algorithm for generating $\Sigma$}
\label{alg:sigma}
\end{algorithm}

For the covariance matrix estimation problem, we computed the Percentage Reduction In Average Loss (PRIAL) with respect to $\frac{n}{n-1}S$ in invariant squared loss $L(\hat\Sigma,\Sigma)=\Tr[(\hat{\Sigma}\Sigma^+-I_p)^2]$ for four estimators. The first three are the estimators $S$, $\hat{\Sigma}_{\text{HF}1}=\frac{n}{n+r}S$ and $\hat{\Sigma}_{\text{HF}2}=\frac{n}{n+r}\big[S+\frac{r-1}{n-r+2}SS^+\Tr^{-1}(S^+)\big]$ considered in Subsection \ref{subsec:cov}. We also included as fourth estimator the diagonal of the sample covariance matrix $\text{diag}(S)$.  The simulation results are given in Figure \ref{fig:cov}. We notice that $\hat{\Sigma}_{\text{HF}1}$ and $\hat{\Sigma}_{\text{HF}2}$ behave similarly, and both improve substantially on $S$, while the diagonal estimator does much worse.

Similarly, for the precision matrix estimation problem, we estimated the PRIAL with respect to  $S^+$ in the Frobenius loss $L(\hat\Sigma^+,\Sigma^+)=\|\hat\Sigma^+-\Sigma^+\|_F^2$ for four estimators. The first three are the estimators $\frac{n-r-2}{n}S^+$, $\hat{\Sigma}_{\text{EM}1}=\frac{n-r-4}{n}S^+$ and $\hat{\Sigma}_{\text{EM}2}= \frac{n-r-4}{n}\big[S^++2\frac{r-1}{n-r-4}SS^+\Tr^{-1}(S)\big]$ from Subsection \ref{subsec:pres}. The fourth one is the inverse of the diagonal of the sample covariance matrix, $\text{diag}(S)^{-1}$. The simulation results are given in Figure \ref{fig:pres}. We can see that all first three estimators improve substantially over $S^+$, but do not differ significantly in risk. In contrast, the diagonal estimator performs well when the true matrix is almost full rank, but becomes worse and worse for smaller covariance ranks.

Finally, for the discriminant coefficient estimation problem, we estimated the PRIAL with respect to $S^+\bar X$ in the square loss $L(\hat\eta,\eta)=\|\hat\eta-\eta\|_2^2$ for four estimators. The first three estimators are $\frac{n-r-2}nS^+\bar X$, $\hat\eta_{\text{TK}1}=\frac{n-r-3}nS^+\bar X$ and $\hat\eta_{\text{TK}2}=\frac{n-r-3}{n}\big[S^+-\frac{r+1}{n-r-3}\Tr^{-1}(S)\big]\bar X$, which were considered in Subsection \ref{subsec:disc}. The fourth one is the estimator $\text{diag}(S)^{-1}\bar X$, which has been considered in linear discriminant analysis when $p>n$. The simulation results are given in Figure \ref{fig:disc}. In this case again, all first three estimators have similar risk and substantially improve on the naive estimator, $S^+\bar X$, while the diagonal estimator is acceptable only when the true covariance matrix is almost full rank and quite bad otherwise.

\begin{figure}[t]\centering
\begin{subfigure}[b]{0.49\textwidth}\centering
\includegraphics[width=\textwidth,clip,trim=0pt 10pt 10pt 50pt]{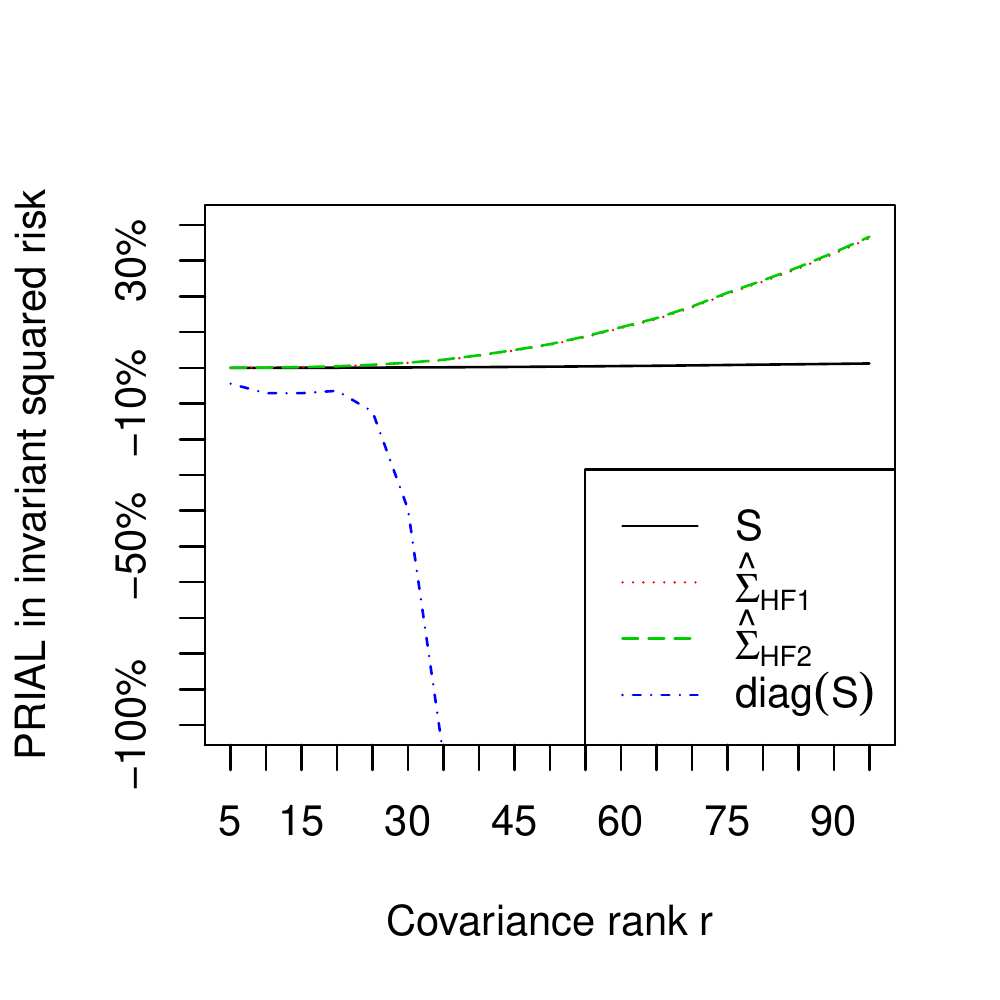}
\caption{n=150 and p=100}
\label{fig:cov-a}
\end{subfigure}
\begin{subfigure}[b]{0.49\textwidth}\centering
\includegraphics[width=\textwidth,clip,trim=0pt 10pt 10pt 50pt]{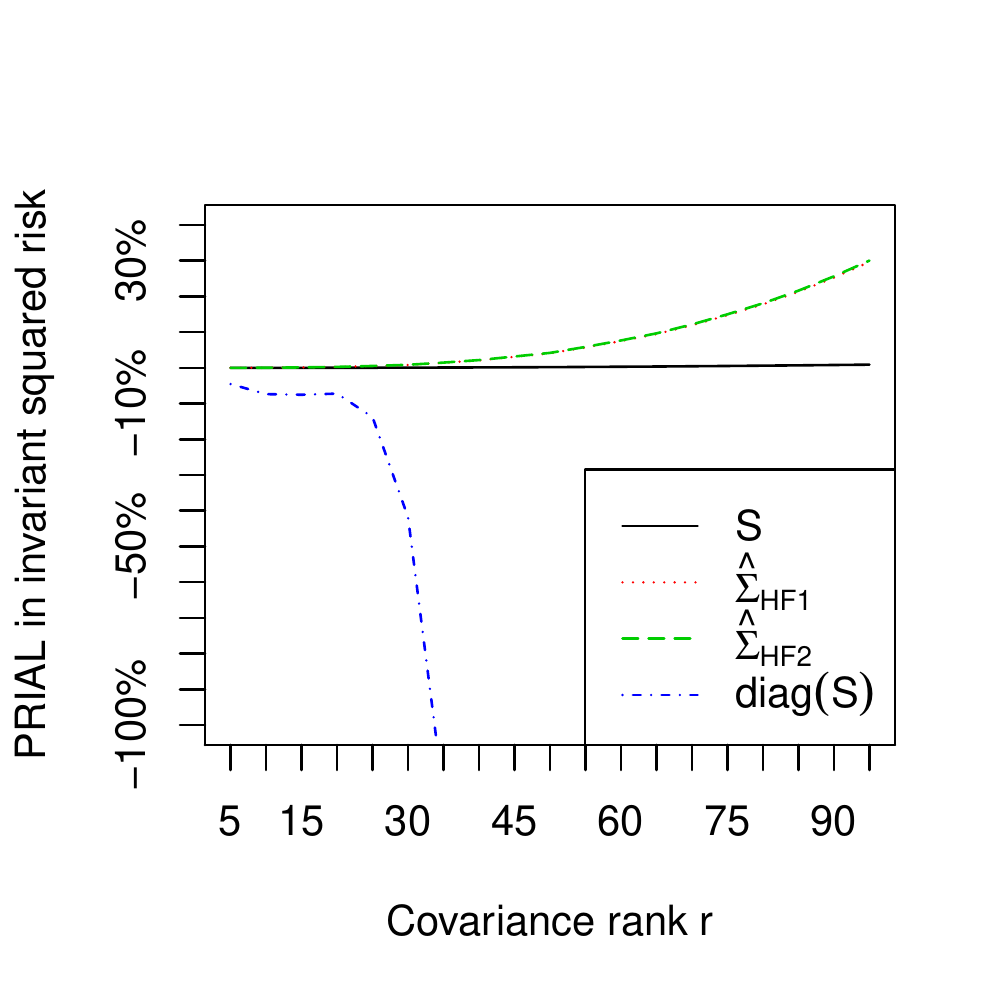}
\caption{n=200 and p=100}
\label{fig:cov-b}
\end{subfigure}
\\
\begin{subfigure}[b]{0.49\textwidth}\centering
\includegraphics[width=\textwidth,clip,trim=0pt 10pt 10pt 50pt]{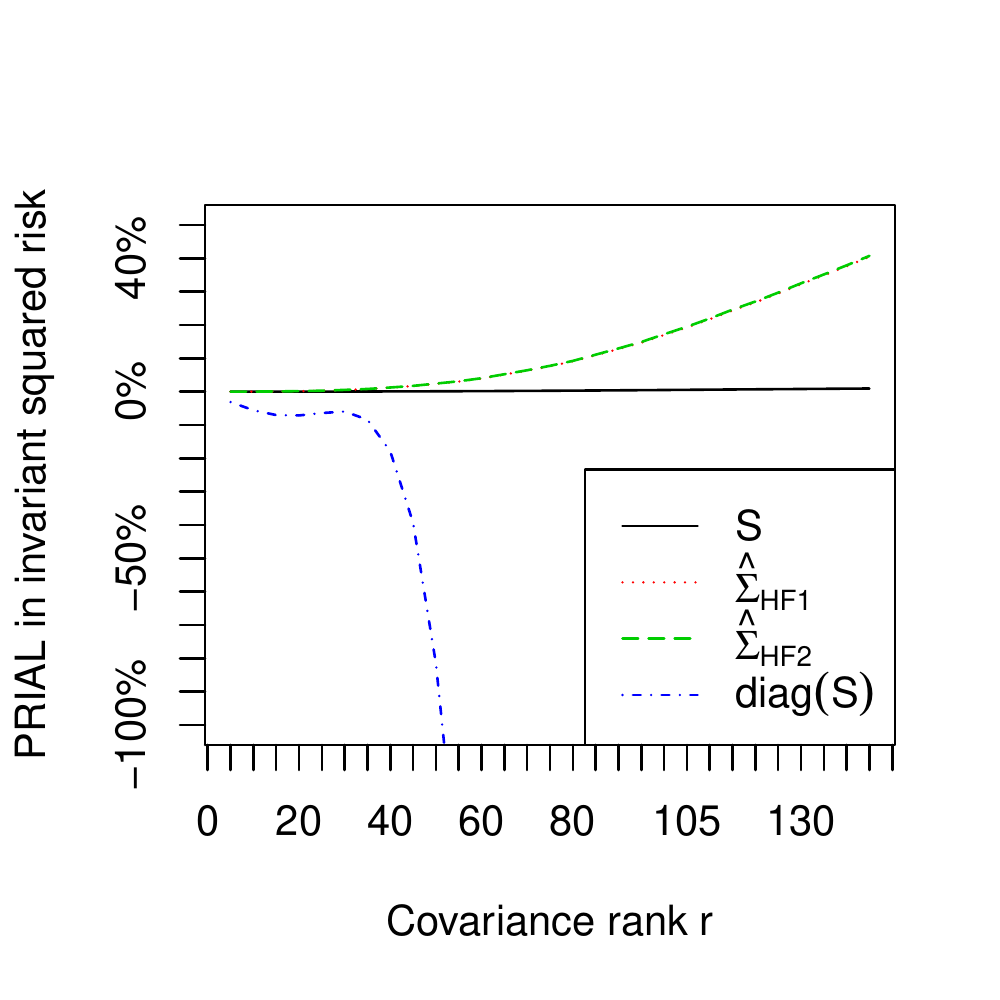}
\caption{n=200 and p=150}
\label{fig:cov-c}
\end{subfigure}
\begin{subfigure}[b]{0.49\textwidth}\centering
\includegraphics[width=\textwidth,clip,trim=0pt 10pt 10pt 50pt]{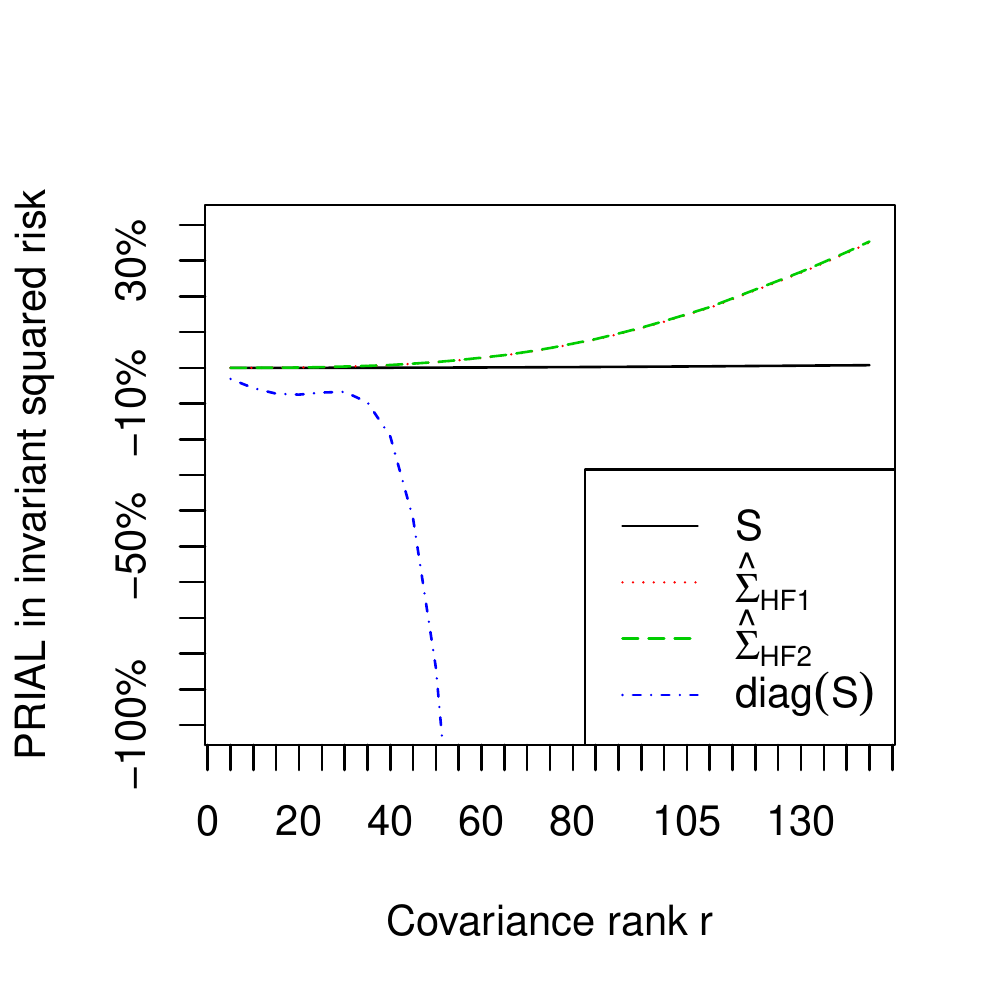}
\caption{n=250 and p=150}
\label{fig:cov-d}
\end{subfigure}
\caption{PRIAL of $S$, $\hat{\Sigma}_{\text{HF}1}$, $\hat{\Sigma}_{\text{HF}2}$ and $\text{diag}(S)$ with respect to $\frac{n-1}nS$ for estimating $\Sigma$ in invariant squared loss.}
\label{fig:cov}
\end{figure}

\begin{figure}[ht]\centering
\begin{subfigure}[b]{0.49\textwidth}\centering
\includegraphics[width=\textwidth,clip,trim=0pt 10pt 10pt 50pt]{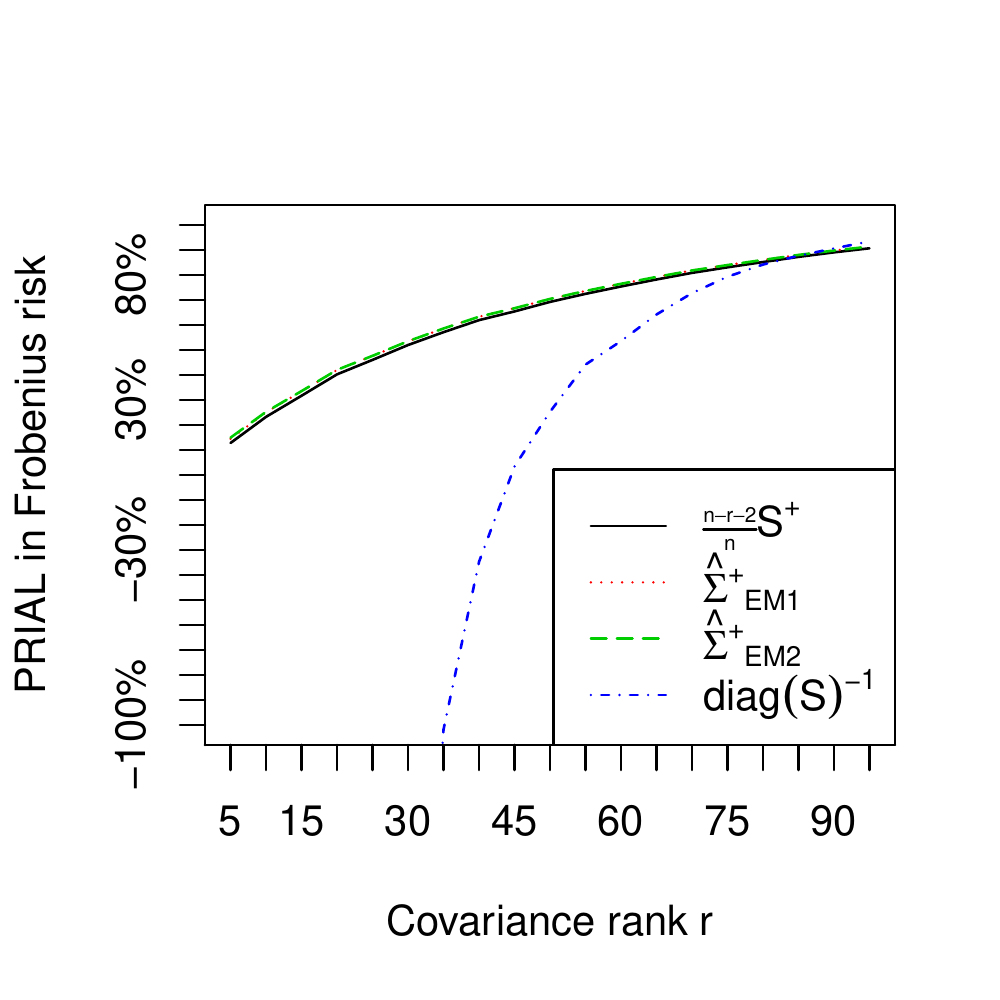}
\caption{n=150 and p=100}
\label{fig:pres-a}
\end{subfigure}
\begin{subfigure}[b]{0.49\textwidth}\centering
\includegraphics[width=\textwidth,clip,trim=0pt 10pt 10pt 50pt]{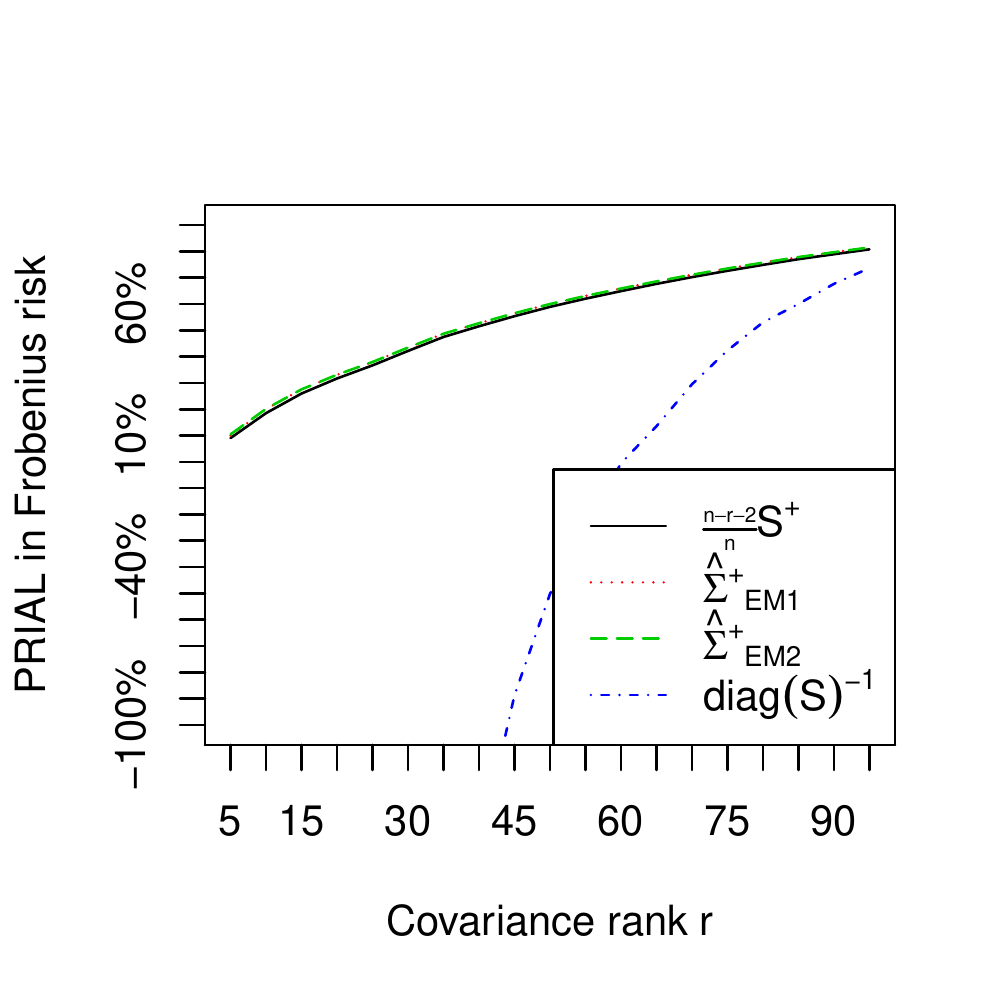}
\caption{n=200 and p=100}
\label{fig:pres-b}
\end{subfigure}
\\
\begin{subfigure}[b]{0.49\textwidth}\centering
\includegraphics[width=\textwidth,clip,trim=0pt 10pt 10pt 50pt]{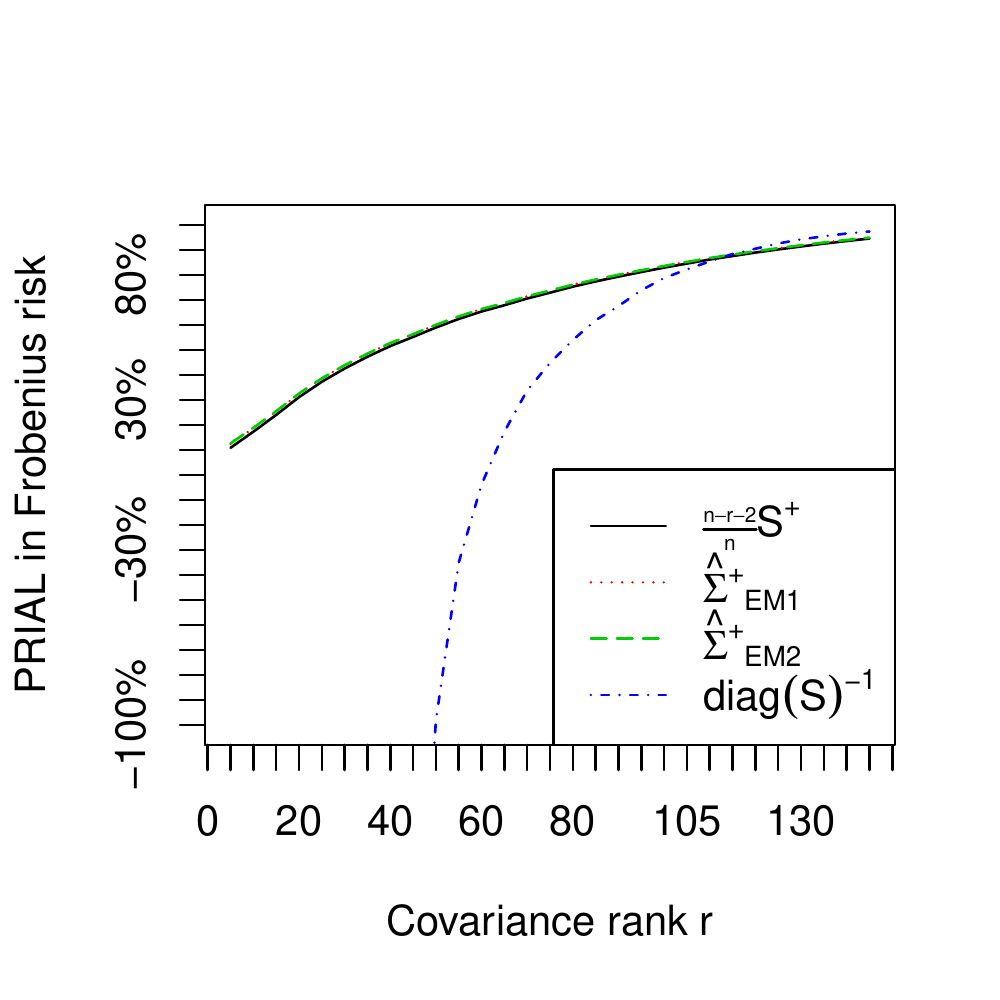}
\caption{n=200 and p=150}
\label{fig:pres-c}
\end{subfigure}
\begin{subfigure}[b]{0.49\textwidth}\centering
\includegraphics[width=\textwidth,clip,trim=0pt 10pt 10pt 50pt]{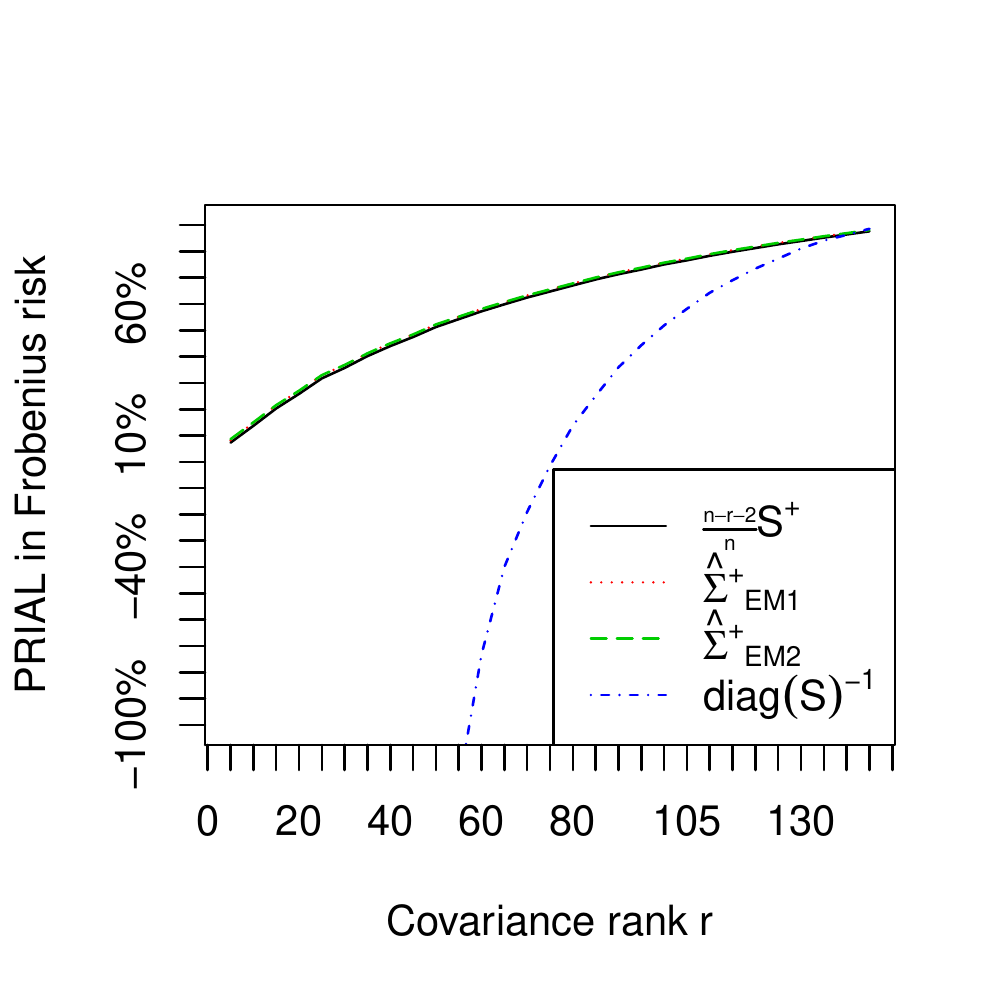}
\caption{n=250 and p=150}
\label{fig:pres-d}
\end{subfigure}
\caption{PRIAL of $\frac{n-r-2}nS^+$, $\hat{\Sigma}^+_{\text{EM}1}$, $\hat{\Sigma}^+_{\text{EM}2}$ and $\text{diag}(S)^{-1}$ with respect to $S^+$ for estimating $\Sigma^+$ in Frobenius loss.}
\label{fig:pres}
\end{figure}

\begin{figure}[ht]\centering
\begin{subfigure}[b]{0.49\textwidth}\centering
\includegraphics[width=\textwidth,clip,trim=0pt 10pt 10pt 50pt]{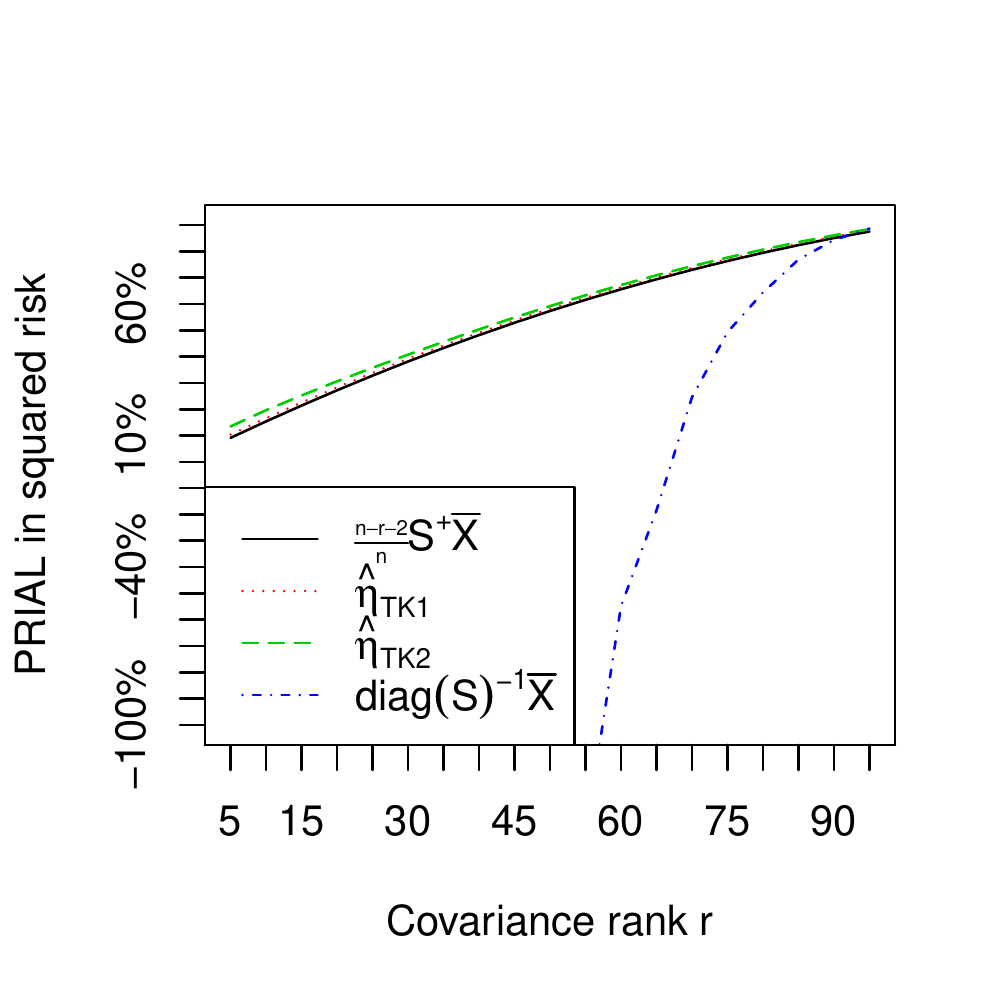}
\caption{n=150 and p=100}
\label{fig:disc-a}
\end{subfigure}
\begin{subfigure}[b]{0.49\textwidth}\centering
\includegraphics[width=\textwidth,clip,trim=0pt 10pt 10pt 50pt]{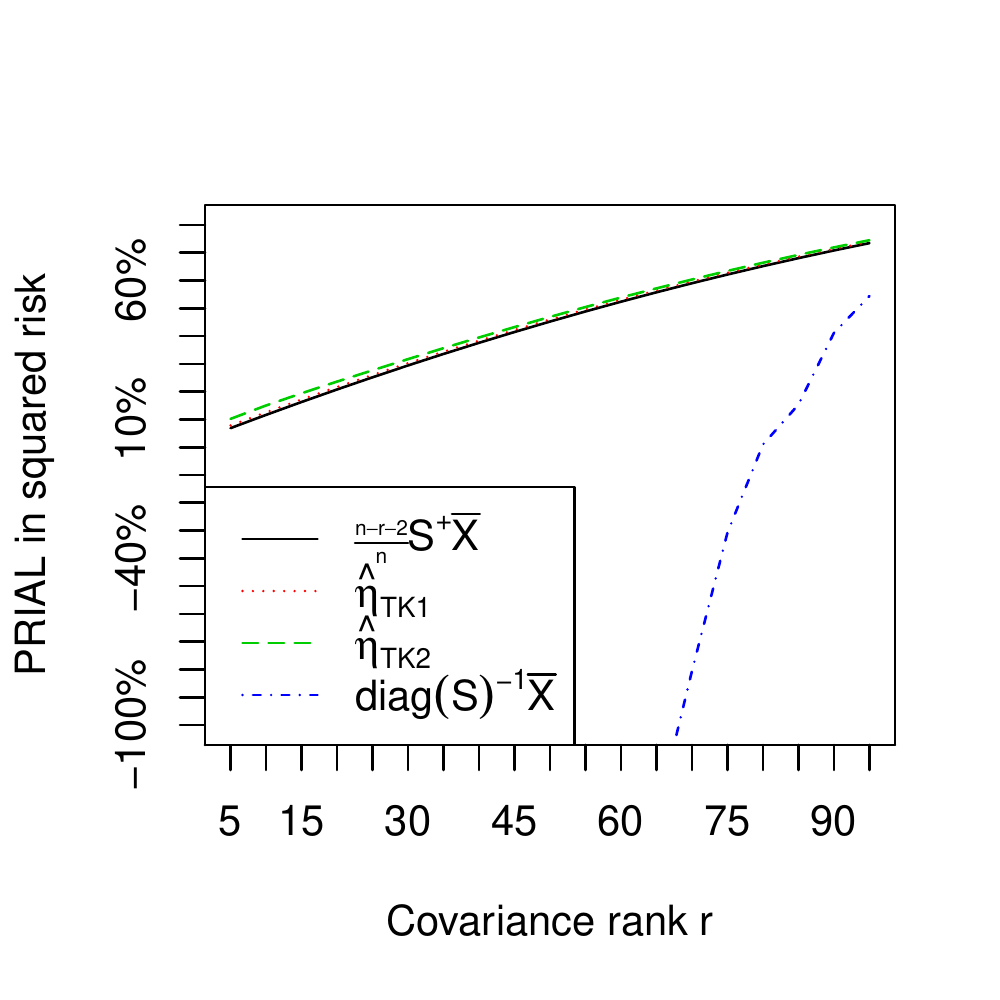}
\caption{n=200 and p=100}
\label{fig:disc-b}
\end{subfigure}
\\
\begin{subfigure}[b]{0.49\textwidth}\centering
\includegraphics[width=\textwidth,clip,trim=0pt 10pt 10pt 50pt]{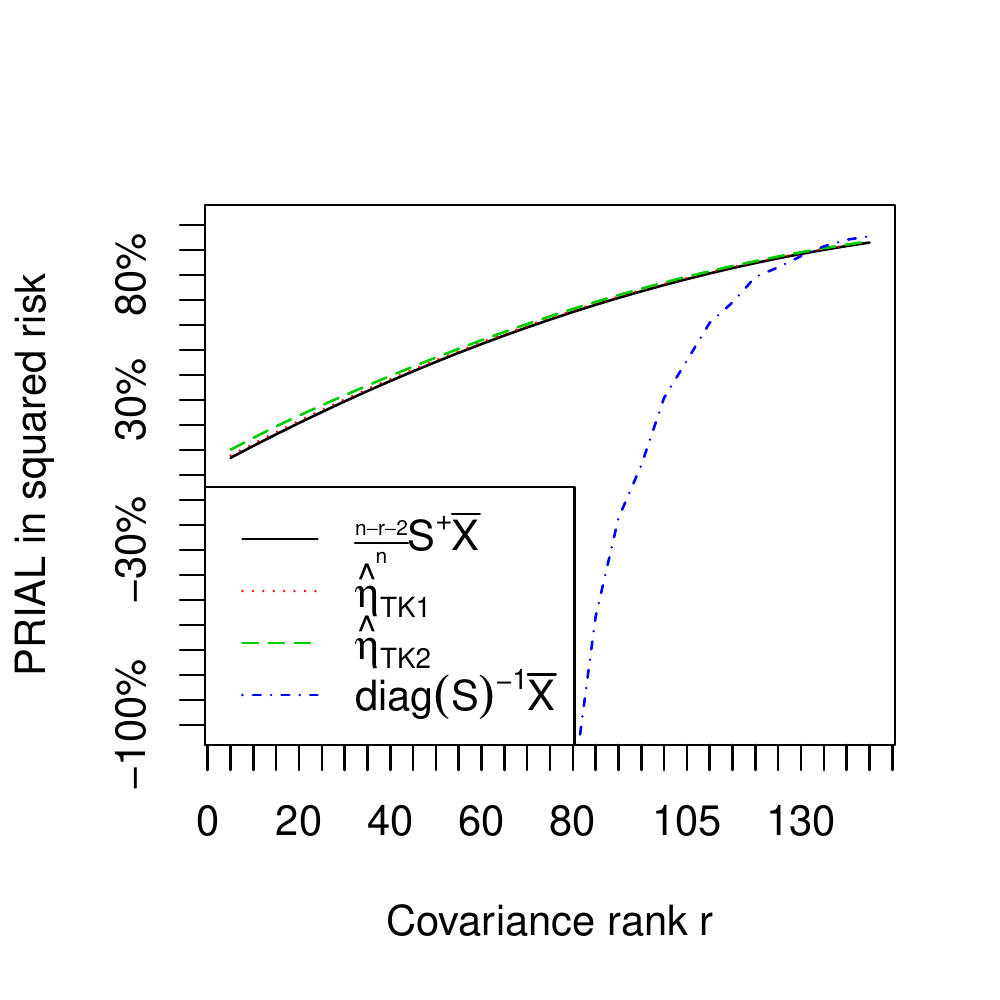}
\caption{n=200 and p=150}
\label{fig:disc-c}
\end{subfigure}
\begin{subfigure}[b]{0.49\textwidth}\centering
\includegraphics[width=\textwidth,clip,trim=0pt 10pt 10pt 50pt]{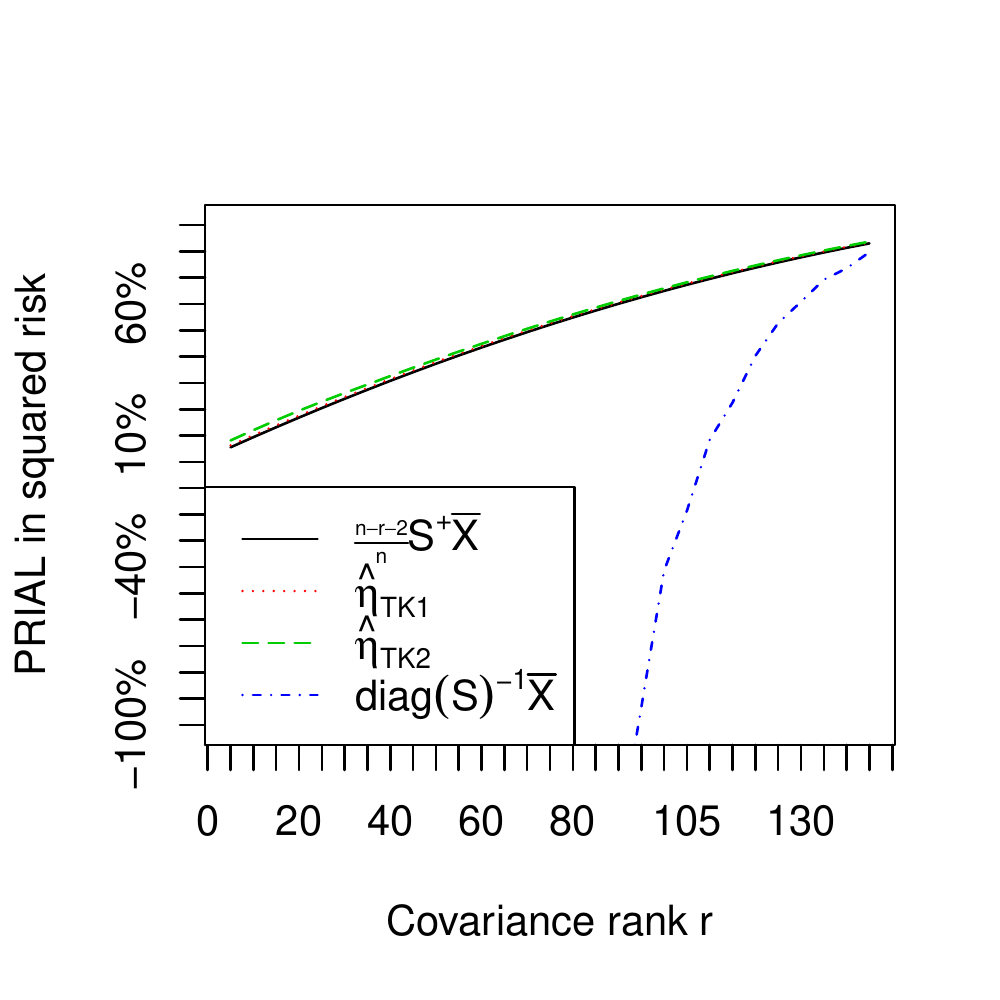}
\caption{n=250 and p=150}
\label{fig:disc-d}
\end{subfigure}
\caption{PRIAL of $\frac{n-r-2}nS^+\bar X$, $\hat{\eta}^+_{\text{TK}1}$, $\hat{\eta}^+_{\text{TK}2}$ and $\text{diag}(S)^{-1}\bar X$ with respect to $S^+\bar X$ for estimating $\eta=\Sigma^+\mu$ in squared loss.}
\label{fig:disc}
\end{figure}

\subsection{NASDAQ-100 simulation}

To explore more realistic designs than an autoregressive covariance matrix, we also considered a setting where the true covariance matrix was constructed from real data.

The NASDAQ-100 is a stock market index composed of the hundred largest non-financial companies on the NASDAQ. As of 2015, this is composed of 107 securities, since some companies offer several classes of stock. We computed the net daily returns of these assets up to March 6, 2015. The newest security is Liberty Media Corp Series C (LMCK), which was issued to series A and B shareholders as dividend on July 7, 2014. To avoid missing data issues, we took this date as the initial time point. This yielded a sample size of 167 trading days. From this data we computed a $107\times107$ sample covariance matrix of the NASDAQ-100 returns.

We then proceeded with the risk simulation as follows. For every $r$ from 1 to $(n-4)\wedge p$, the true covariance matrix $\Sigma$ was defined as the NASDAQ-100 sample covariance matrix with its $p-r$ smallest eigenvalues set to zero. We then randomly generated $1,000$ replications from a multivariate normal distribution with mean $\mu=(1,\dots,1)$ and singular covariance $\Sigma$, and computed the resulting sample covariance matrix $S=X^tX/n$.

For each of the three estimation problems, we computed the PRIAL as in Subsection \ref{subsec:AR}. The simulation results are given in Figure \ref{fig:NASDAQ}. The results appear similar to the singularized autoregressive setting.

\begin{figure}[ht]\centering
\begin{subfigure}[b]{0.49\textwidth}\centering
\includegraphics[width=\textwidth,clip,trim=0pt 10pt 10pt 50pt]{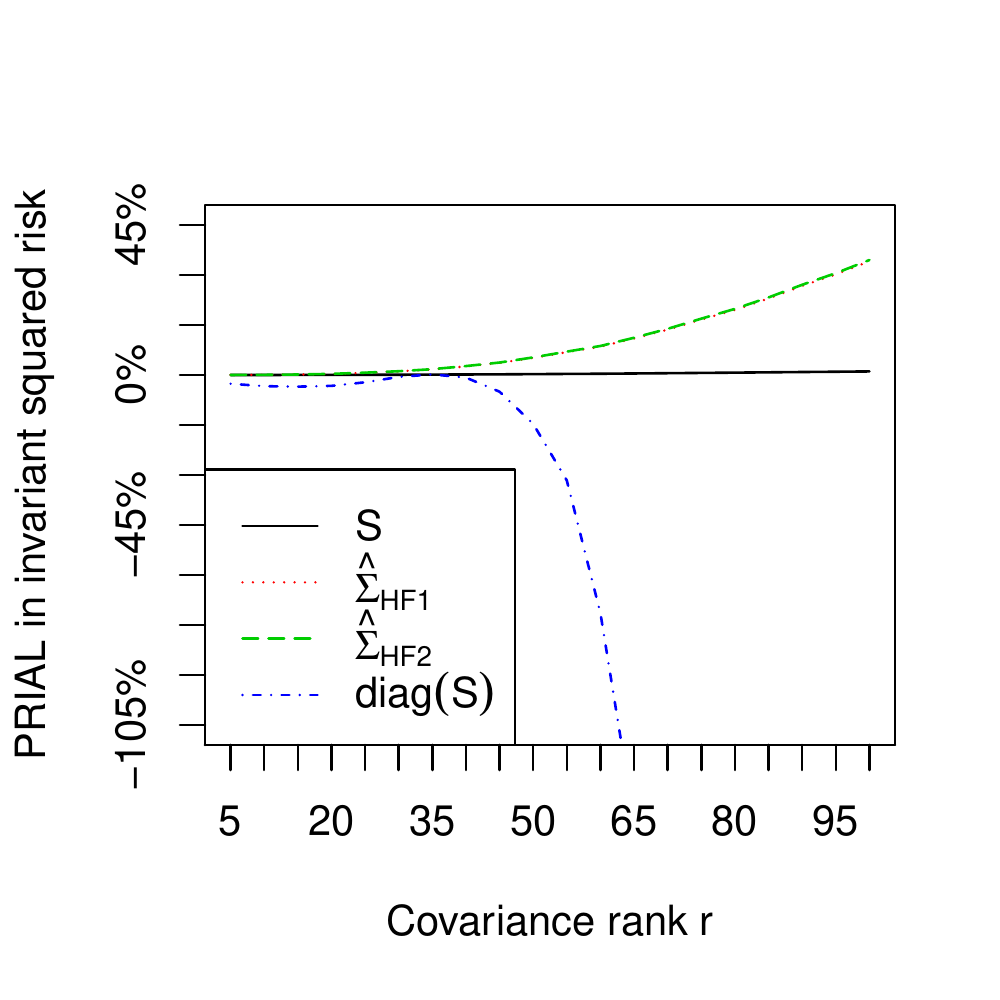}
\caption{Covariance matrix estimation}
\label{fig:NASDAQ-a}
\end{subfigure}
\begin{subfigure}[b]{0.49\textwidth}\centering
\includegraphics[width=\textwidth,clip,trim=0pt 10pt 10pt 50pt]{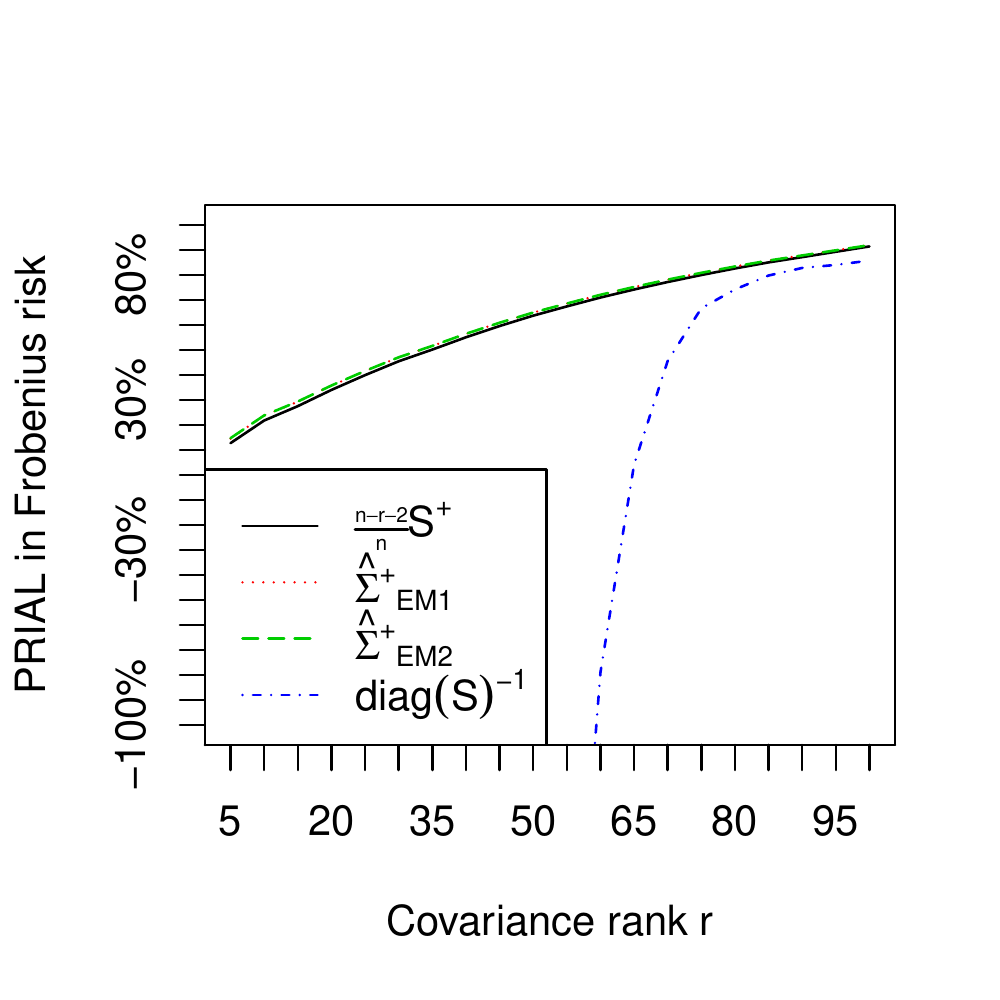}
\caption{Precision matrix estimation}
\label{fig:NASDAQ-b}
\end{subfigure}
\\
\begin{subfigure}[b]{0.49\textwidth}\centering
\includegraphics[width=\textwidth,clip,trim=0pt 10pt 10pt 50pt]{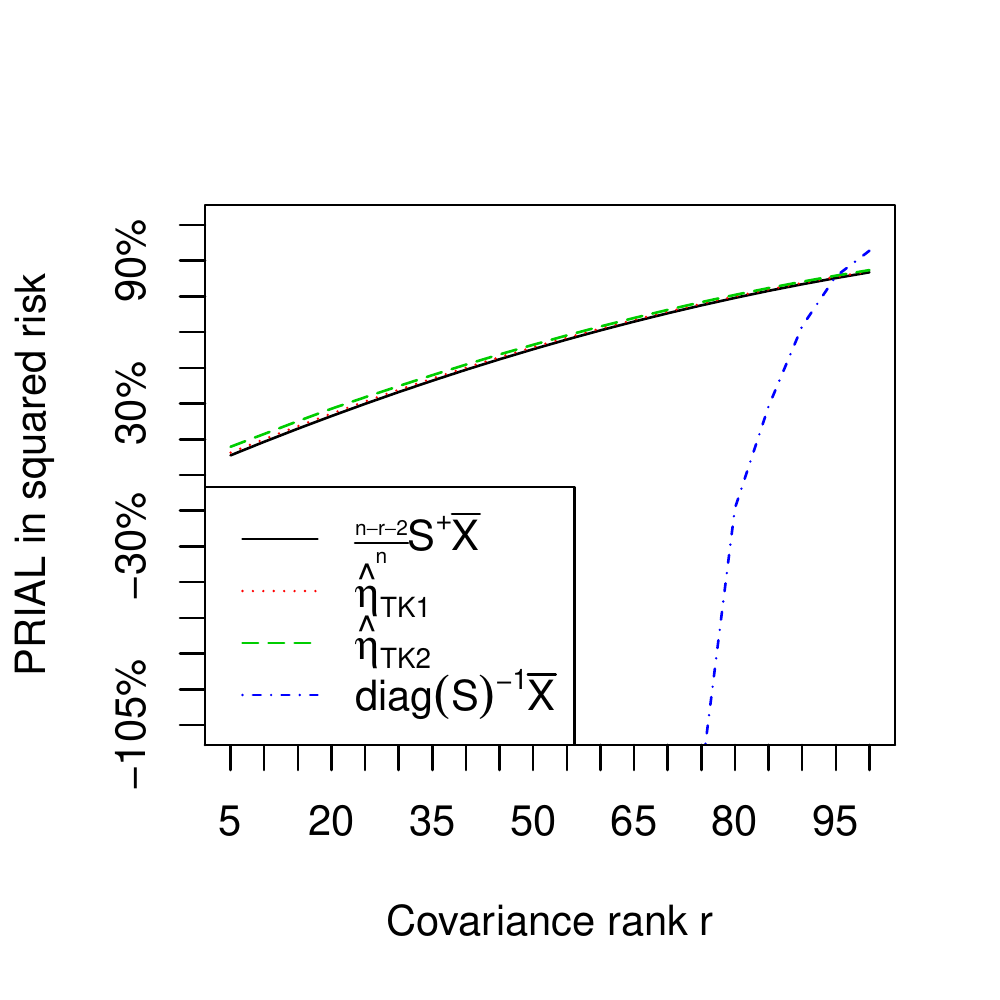}
\caption{Discriminant coefficient estimation}
\label{fig:NASDAQ-c}
\end{subfigure}
\caption{PRIAL for the singularized NASDAQ-100 covariance matrix in the three estimation tasks.}
\label{fig:NASDAQ}
\end{figure}

\section{Discussion}

An application of the \cite{TsukumaKubokawa14} technique developed in Subsection \ref{subsec:model} allows in essence to reduce the dimension from $p$ to $r$. Since $r<\min(n,p)$, this in effect turns the problem into a classical setting where the sample size is greater than the dimension, and allows for classical proof techniques to be applied.

An interesting extension is the setting where $n\leq r<p$. In that case, an adaptation of the method would yield a high-dimensional context where the true covariance matrix is full rank, but the sample size $n$ is still smaller than the dimension $p$. Recent work, for example by \citet{Konno09}, could allow the construction of improved estimators analogous to the ones presented in this article.

Recent attention has been given to the notion of the effective rank of a matrix $r(A)=\Tr(A)/\|A\|_2$, developed by \cite{Vershynin10}and applied in the study of spiked covariance matrices in \cite{Bunea12}. Singular covariance matrices can be regarded as a boundary case of spiked matrices where the noise equals zero. In that regard, it is interesting to notice that the quantity $\Tr(S^{+2})/\Tr^2(S^+)$ that appear in inequality  \eqref{eq:cov-aS+tTr-risk} is related to the effective rank of $S^+$ through the inequality
\begin{align*}
\frac{\Tr(S^{+2})}{\Tr^2(S^+)}\leq
r(S^+)
\leq r^2\frac{\Tr(S^{+2})}{\Tr^2(S^+)}.
\end{align*}
The presence of this quantity is likely connected to the orthogonal invariance of the loss function.

The some of the key results in \cite{bickel2004} can be extended to the case where $\rk{\Sigma} = r< p$.  Suppose $\Sigma$ is known and let $e_1$ and $e_2$  equal the limiting Bayes risk of the classification rule using $\Sigma^+$ and ${\diag} (\Sigma)$, respectively.   By an application of an extended Kantorovich inequality for generalized inverses, developed by \cite{Liu97}, it can be shown that $e_2 \leq {\bar \Phi}(\frac{2\sqrt \kappa_r}{1+\kappa_r} {\bar \Phi}(e_1))$, where ${\bar \Phi}$ is the Gaussian survival function and $\kappa_r = \lambda_1 / \lambda_r$ with $\lambda_1 > \lambda_2 >  \cdots > \lambda_r$ the non-zero eigenvalues of $({\diag}( \Sigma ))^{-1}\Sigma$. In the setting of \cite{bickel2004}, $\Sigma$ is assumed to be full rank so that the limiting Bayes risk of the classification rule using $\diag (\Sigma) $ is close to optimal.  However, in the rank deficient case $\kappa_q = \infty$ for $q > r$, which implies that the diagonal rule give rise to a procedure that is no better than random guessing, that is with $e_2=1/2$.  This behavior is evident from Figures  \ref{fig:disc} and \ref{fig:NASDAQ}. In the case where the rank is close to $p$, the risk of the diagonal based discriminant estimator is close to the improved estimates, however, as the rank of $\Sigma$ declines from $p$ the risk properties of the diagonal based discriminant estimator become inferior.

Finally, in applications where a singular covariance matrix is unlikely but a low-dimensional approximation is desired, it might be beneficial to use one of the estimators proposed in this article and cross-validate the rank $r$ on the task to accomplish. For example, a mean-variance portfolio optimization problem could use $\hat{\Sigma}^+_\text{EM2}$ as precision matrix estimate, with rank $r$ cross-validated on some validation set. To the best of our knowledge, this methodology has no theoretical grounding but might nevertheless prove useful in some high-dimensional problems.

\section{Proofs}\label{sec:proofs}

\subsection{Preliminaries}\label{subsec:prelim} Before presenting the proofs of the statements from Section \ref{sec:results}, we explain the techniques employed by \citet{TsukumaKubokawa14} to work around the singularity of the covariates in the model. Define the sample mean and covariance matrix to be
\begin{align*}
\begin{array}{l l}
\bar X=X'1_n/n   &\sim\text{N}_p(\mu,\Sigma/n),
\\
S=[X-1_n\bar X']'[X-1_n\bar X']/n  &\sim\text{W}_p(n-1,\Sigma/n).
\end{array}
\end{align*}

Since $\Sigma$ has rank $r$, we can factorize it as $\Sigma=BB^t$ for some full rank $p\times r$ matrix $B$. Let $H=B(B^tB)^{-1/2}$ and $\Omega=B^tB$ - then $H$ is $p\times r$ semi-orthogonal $H^tH=I_r$ and $HH^t=\Sigma\Sigma^+$, $\Omega$ is $r\times r$ invertible, $\Sigma=H\Omega H^t$ and $\Sigma^+=H\Omega^{-1}H^t$. Since $\Sigma$ is rank deficient, there must be a $Z\sim\text{N}_{n,r}(0,I_r)$ such that $X=1_n\mu^t+ZB^t$, and therefore we can write $X=1_n\mu^t+Z(B^tB)^{1/2}(B^tB)^{-1/2}B^t=1_n\mu^t+YH^t$ for $Y=Z\Omega^{-1/2}\sim N_{n,r}(0,\Omega)$. Define then
\begin{align*}
\begin{array}{l l}
\bar Y=Y^t1_n/n&\sim\text{N}_r(0,\Omega/n),
\\
T=[Y-1_n\bar Y^t]'[Y-1_n\bar Y^t]/n&\sim\text{W}_r(n-1,\Omega/n).
\end{array}
\end{align*}
Notice how $T$ is full rank, since $r\leq n-1$, in contrast with $S$. Using $X=1_n\mu^t+YH^t$, we can see that these constructions are related to $\bar X$ and $S$ through
\begin{align*}
\bar X &= \mu+H\bar Y,
\hspace{60pt}
S = HTH^t.
\end{align*}

Recall that $SS^+=\Sigma\Sigma^+$ almost surely, from Equation \eqref{eq:SSSS}. Since $S$ has rank $r<p$, there must be a $p\times r$ semi-orthogonal matrix $O_1$ such that $O_1^tO_1=I_r$, $O_1O_1^t=\Sigma\Sigma^+$ almost surely and $S=O_1LO_1^t$ for $L=\text{diag}(\lambda_1(S),...,\lambda_r(S))$. The $r\times r$ matrix $U=H^tO_1$ is easily seen to be orthogonal, and so by $T=H^tSH=H^tO_1LO_1H^t=ULU^t$, we see that $T$ and $S$ must share the same $r$ non-zero eigenvalues, i.e. $\lambda_i(S)=\lambda_i(T)$.

These constructions and facts form the basis of our risk estimation procedures and the notation will be repeatedly used in the following subsections.

\subsection{Proofs of Subsection \ref{subsec:cov}}
\begin{proof}[Proof of Theorem \ref{thm:cov-ure}] Since $T$ and $S$ share the same non-zero eigenvalues, we can regard $\Psi$ as a function of $T\sim\text{W}_r(n-1,\Omega/n)$ only.
Since $r\leq n-1$ and $\Omega$ is full rank, we can apply Lemma 1 and 2 of \cite{ChetelatWells14}  to $H^t\hat{\Sigma}H=U\Psi U^t$. On that result, one can also consult \citet[Theorem 4.1]{Sheena95}, and in the singular case \citet[Proposition 2.1]{KubokawaSrivastava08} and \citet[Theorem 2.4]{Konno09}. In any case, we get
{\setlength{\mathindent}{5pt}\begin{align*}&
\E{\Tr\Big(\big[\hat\Sigma\Sigma^+-I_p\big]^2\Big)}
=\E{p-2\tr{\Sigma^+\hat\Sigma}+\tr{\Sigma^+\hat\Sigma\Sigma^+\hat\Sigma}}
\\&\qquad
=\E{(p-r)+r-2\tr{\Omega^{-1}U\Psi U'}+\tr{\Omega^{-1}U\Psi U'\Omega^{-1}U\Psi U'}}
\\&\qquad
=\E{p-r+\Tr\Big(\big[U\Psi U'\Omega^{-1}-I_r\big]^2\Big)}
\\&\qquad
=\E{(p-r)+r
+\frac{n-r-2}{n}\sum_{k=1}^r\frac{\psi^*_k}{l_k}
+\frac{2}{n}\sum_{k=1}^r\frac{\partial\psi^*_k}{\partial l_k}
+\frac1{n}\sum_{k\not= b}^r\frac{\psi^*_k-\psi^*_b}{l_k-l_b}
},
\end{align*}}
under the regularity conditions
{\setlength{\mathindent}{5pt}\begin{align*}
&\E{\left|\sum_{k=1}^r \frac{n-r-2}n\frac{\psi_k}{l_k}
+\frac2n\sum_{k=1}^r\frac{\partial \psi_k}{\partial l_k}
+\frac1n\sum_{k\neq b}^r\frac{\psi_k-\psi_b}{l_k-l_b}\right|}<\infty,
\\&
\E{\left|
\sum_{k=1}^r\frac{n-r-2}n\frac{\psi_k^*+2\psi_k}{l_k}
+\frac2n\sum_{k=1}^r\frac{\partial \psi_k^*+2\psi_k}{\partial l_k}
\right.\right.\\&\hspace{20pt}\left.\left.
+\frac1n\sum_{k\neq b}^r\frac{\psi_k^*+2\psi_k-\psi_b^*-2\psi_b}{l_k-l_b}
\right|}<\infty
\text{ and }
\E{\sum_{k=1}^r\left|\frac{\psi^*_k+2\psi_k}{l_k}\right|^2}<\infty.
\end{align*}}
But these are satisfied by Inequalities \eqref{eq:cov-ure-conditions}. This concludes the proof.
\end{proof}

\begin{proof}[Proof of Proposition \ref{prop:cov-aS}] Let us apply the results of Theorem \ref{thm:cov-ure}. We have $\psi_k=a l_k$, so
\begin{align*}
\psi^*_k&=\left[\frac{n-r-2}na+\frac4na+\frac2na(r-1)-2\right]al_k
\\&
=\left[\frac{n+r}{n}a-2\right]al_k.
\end{align*}
Then the unbiased risk estimator \eqref{eq:cov-ure-ure} equals
{\setlength{\mathindent}{20pt}\begin{align*}
U&=p
+\frac{n-r-2}{n}\sum_{k=1}^r\frac{\psi^*_k}{l_k}
+\frac{2}{n}\sum_{k=1}^r\frac{\partial\psi^*_k}{\partial l_k}
+\frac1{n}\sum_{k\not= b}^r\frac{\psi^*_k-\psi^*_b}{l_k-l_b}
\\&
=p
+\frac{n-r-2}{n}\left[\frac{n+r}{n}a-2\right]ar
+\frac{2}{n}\left[\frac{n+r}{n}a-2\right]ar
\\&\hspace{40pt}
+\frac1{n}\left[\frac{n+r}{n}a-2\right]ar(r-1)
\\&
=p
-2\frac{(n-1)r}{n}a
+\frac{(n-1)(n+r)r}{n^2}a^2.
\end{align*}}
Clearly, $\E{\Big|U\Big|}=\Big|p-2\frac{(n-1)r}{n}a
+\frac{(n-1)(n+r)r}{n^2}a^2\Big|<\infty$. Similarly,
{\setlength{\mathindent}{20pt}\begin{align*}&
\E{\left|p
+\sum_{k=1}^r\frac{n-r-2}n\frac{\psi_k}{l_k}
+\frac2n\sum_{k=1}^r\frac{\partial \psi_k}{\partial l_k}
+\frac1n\sum_{k\neq b}^r\frac{\psi_k-\psi_b}{l_k-l_b}
\right|}
\\&\qquad
=
\E{\left|p
+\frac{(n-r-2)r}na
+\frac{2r}na
+\frac{r(r-1)}na
\right|}
\\&\qquad
=\left|p+\frac{(n-1)r}na\right|<\infty,
\\&
\E{\sum_{k=1}^r\left|\frac{\psi^*_k}{l_k}\right|^2}
=r\left[\frac{n+r}{n}a-2\right]^2a^2<\infty,
\quad
\E{\sum_{k=1}^r\left|\frac{\psi_k}{l_k}\right|^2}
=ra^2<\infty.
\end{align*}}
Thus the regularity conditions of Theorem \ref{thm:cov-ure} are satisfied and
{\setlength{\mathindent}{20pt}\begin{align*}
\E{\Tr\Big(\big[\hat\Sigma\Sigma^+-I_p\big]^2\Big)}=\E{U}
=p
-2\frac{(n-1)r}{n}a
+\frac{(n-1)(n+r)r}{n^2}a^2.
\end{align*}}
But this is minimized when $a=\frac{n}{n+r}$. In particular, notice that since $n\geq r+1=2$,
{\setlength{\mathindent}{20pt}\begin{align*}
\E{\Tr\Big(\big[\hat\Sigma_\text{HF1}\Sigma^+-I_p\big]^2\Big)}
&=p-\frac{(n-1)r}{n+r}
\\&\hspace{-20pt}
<
p-\frac{(n-r)(n-1)r}{n^2}
=
\E{\Tr\Big(\big[S\Sigma^+-I_p\big]^2\Big)}
\\&\hspace{-20pt}
<
p-\frac{(n-r-2)r}{n-1}
=
\E{\Tr\Big(\big[\frac{n}{n-1}S\Sigma^+-I_p\big]^2\Big)},
\end{align*}}
so $\hat\Sigma_\text{HF1}$ dominates $S$, which dominates $\frac{n}{n-1}S$, as desired.
\end{proof}

\begin{proof}[Proof of Proposition \ref{prop:cov-aS+tTr}] Again, let us apply the results of Theorem \ref{thm:cov-ure}. Here $\psi_k = \frac{n}{n+r}[l_k+t/\text{tr}(S^+)]$,
so using that $\frac{\partial}{\partial l_k}\frac1{\text{tr}(S^+)}=\frac1{l^2_k\Tr^2(S^+)}$ we find
{\setlength{\mathindent}{5pt}\begin{align*}&
\psi^*_k
=\left[\frac{n-r-2}n\frac{\psi_k}{l_k}
+\frac4n\frac{\partial \psi_k}{\partial l_k}
+\frac2n\sum_{b\neq k}^r\frac{\psi_k-\psi_b}{l_k-l_b}-2\right]\psi_k
\\&\quad
=\frac{n}{n+r}\bigg[
\frac{n-r-2}{n+r}\left(1+\frac{t}{l_k\text{tr}\,(S^+)}\right)
+\frac4{n+r}\left(1+\frac{t}{l_k^2\text{tr}^2(S^+)}\right)
\\&\hspace{100pt}
+\frac{2(r-1)}{n+r}-2\bigg]
\cdot\left[l_k+\frac{t}{\text{tr}(S^+)}\right]
\\&\quad
=\frac{n}{n+r}\left[1
+\frac{n-r-2}{n+r}\frac{t}{l_k\text{tr}\,(S^+)}
+\frac4{n+r}\frac{t}{l_k^2\text{tr}^2(S^+)}
-2
\right]
\left[l_k+\frac{t}{\text{tr}(S^+)}\right]
\\&\quad
=-\frac{n}{n+r}l_k
+\bigg[
-2\frac{r+1}{n+r}\frac1{\text{tr}\,(S^+)}
+\frac4{n+r}\frac1{l_k\text{tr}^2(S^+)}
\bigg]\frac{nt}{n+r}
\\&\quad\qquad
+\bigg[\frac{n-r-2}{n+r}\frac1{l_k\text{tr}^2(S^+)}
+\frac4{n+r}\frac1{l_k^2\text{tr}^3(S^+)}
\bigg]\frac{nt^2}{n+r}.
\end{align*}}
Let us now compute the terms in the URE. We find for the first term:
{\setlength{\mathindent}{5pt}\begin{align*}&
\frac{n-r-2}{n}\sum_{k=1}^r\frac{\psi^*_k}{l_k}
=-\frac{n-r-2}{n}\sum_{k=1}^r\frac{n}{n+r}
\\&\quad\qquad
+\frac{n-r-2}{n}\sum_{k=1}^r\bigg[-2\frac{r+1}{n+r}\frac1{l_k\text{tr}\,(S^+)}
+\frac4{n+r}\frac1{l^2_k\text{tr}^2(S^+)}
\bigg]\frac{nt}{n+r}
\\&\quad\qquad
+\frac{n-r-2}{n}\sum_{k=1}^r\bigg[\frac{n-r-2}{n+r}\frac1{l^2_k\text{tr}^2(S^+)}
+\frac4{n+r}\frac1{l_k^3\text{tr}^3(S^+)}
\bigg]\frac{nt^2}{n+r}
\\&\quad
=-\frac{n-r-2}{n+r}r
+\frac{n-r-2}{n+r}\bigg[
-2\frac{r+1}{n+r}
+\frac4{n+r}\frac{\Tr(S^{+2})}{\text{tr}^2(S^+)}
\bigg]t
\\&\quad\qquad
+\frac{n-r-2}{n+r}\bigg[\frac{n-r-2}{n+r}\frac{\Tr(S^{+2})}{\text{tr}^2(S^+)}
+\frac4{n+r}\frac{\Tr(S^{+3})}{\text{tr}^3(S^+)}
\bigg]t^2.
\end{align*}}
Next, using the fact that
$\frac{\partial}{\partial l_k}\frac1{l_k\text{tr}^2(S^+)}
=-\frac1{l_k^2\text{tr}^2(S^+)}+\frac2{l_k^3\text{tr}^3(S^+)}$
and that
$\frac{\partial}{\partial l_k}\frac1{l_k^2\text{tr}^3(S^+)}
=-\frac2{l_k^3\text{tr}^3(S^+)}+\frac3{l_k^4\text{tr}^4(S^+)}$,
 we find
{\setlength{\mathindent}{5pt}\begin{align*}&
\frac{2}{n}\sum_{k=1}^r\frac{\partial\psi^*_k}{\partial l_k}
=-\frac{2}{n}\sum_{k=1}^r\frac{\partial}{\partial l_k}\frac{n}{n+r}l_k
\\&\quad\qquad
+\frac{2}{n}\sum_{k=1}^r\frac{\partial}{\partial l_k}
\bigg[-2\frac{r+1}{n+r}\frac1{\text{tr}\,(S^+)}
+\frac4{n+r}\frac1{l_k\text{tr}^2(S^+)}
\bigg]\frac{nt}{n+r}
\\&\quad\qquad
+\frac{2}{n}\sum_{k=1}^r\frac{\partial}{\partial l_k}
\bigg[\frac{n-r-2}{n+r}\frac1{l_k\text{tr}^2(S^+)}
+\frac4{n+r}\frac1{l_k^2\text{tr}^3(S^+)}
\bigg]\frac{nt^2}{n+r}
\\&
=-\frac{2}{n}\sum_{k=1}^r\frac{n}{n+r}
+\frac{2}{n}\sum_{k=1}^r
\bigg[-2\frac{r+1}{n+r}\frac1{l^2_k\text{tr}^2\,(S^+)}
-\frac4{n+r}\frac1{l_k^2\text{tr}^2(S^+)}
\\&\quad\qquad
+\frac4{n+r}\frac2{l_k^3\text{tr}^3(S^+)}
\bigg]\frac{nt}{n+r}
+\frac{2}{n}\sum_{k=1}^r
\bigg[-\frac{n-r-2}{n+r}\frac1{l_k^2\text{tr}^2(S^+)}
\\&\quad\qquad
+\frac{n-r-2}{n+r}\frac2{l_k^3\text{tr}^3(S^+)}
-\frac4{n+r}\frac2{l_k^3\text{tr}^3(S^+)}
+\frac4{n+r}\frac3{l_k^4\text{tr}^4(S^+)}
\bigg]\frac{nt^2}{n+r}
\\&
=-\frac{2r}{n+r}
+\frac{2}{n+r}
\bigg[
-2\frac{r+3}{n+r}\frac{\Tr(S^{+2})}{\text{tr}^2\,(S^+)}
+\frac8{n+r}\frac{\Tr(S^{+3})}{\text{tr}^3\,(S^+)}\bigg]t
\\&\;\;
+\frac{2}{n+r}
\bigg[\!-\!\frac{n-r-2}{n+r}\frac{\Tr(S^{+2})}{\text{tr}^2\,(S^+)}
+2\frac{n-r-6}{n+r}\frac{\Tr(S^{+3})}{\text{tr}^3\,(S^+)}
+\frac{12}{n+r}\frac{\Tr(S^{+4})}{\text{tr}^4\,(S^+)}
\bigg]t^2.
\end{align*}}
Finally, using that $\sum_{k\not=b}^r\frac{l_k^{-1}-l_b^{-1}}{l_k-l_b}\leq0$ and $\sum_{k\not=b}^r\frac{l_k^{-2}-l_b^{-2}}{l_k-l_b}\leq0$ we can bound
{\setlength{\mathindent}{5pt}\begin{align*}&
\frac1{n}\sum_{k\not= b}^r\frac{\psi^*_k-\psi^*_b}{l_k-l_b}
=-\frac1{n}\sum_{k\not= b}^r\frac{n}{n+r}
+\frac1{n}
\bigg[
\frac4{n+r}\frac1{\text{tr}^2(S^+)}
\sum_{k\not=b}^r\frac{l_k^{-1}-l_b^{-1}}{l_k-l_b}
\bigg]\frac{nt}{n+r}
\\&\;\;
+\!\frac1{n}
\bigg[\frac{n-r-2}{n+r}\frac1{\text{tr}^2(S^+)}
\sum_{k\not=b}^r\frac{l_k^{-1}-l_b^{-1}}{l_k-l_b}
\!+\!\frac4{n+r}\frac1{\text{tr}^3(S^+)}
\sum_{k\not=b}^r\frac{l_k^{-2}-l_b^{-2}}{l_k-l_b}
\bigg]\frac{nt^2}{n+r}
\\&\quad
\leq\;-\frac{r(r-1)}{n+r}.
\end{align*}}
Hence the URE \eqref{eq:cov-ure-ure} equals
{\setlength{\mathindent}{5pt}\begin{align*}&
U=\;p
+\frac{n-r-2}{n}\sum_{k=1}^r\frac{\psi^*_k}{l_k}
+\frac{2}{n}\sum_{k=1}^r\frac{\partial\psi^*_k}{\partial l_k}
+\frac1{n}\sum_{k\not= b}^r\frac{\psi^*_k-\psi^*_b}{l_k-l_b}
\\&
\leq\;
p-\frac{n-r-2}{n+r}r
-\frac{2}{n+r}r
-\frac{r-1}{n+r}r
\\&\quad
+\frac{n-r-2}{n+r}\bigg[
-2\frac{r+1}{n+r}
+\frac4{n+r}\frac{\Tr(S^{+2})}{\text{tr}^2(S^+)}
\bigg]t
\\&\quad
+\frac{2}{n+r}
\bigg[
-2\frac{r+3}{n+r}\frac{\Tr(S^{+2})}{\text{tr}^2\,(S^+)}
+\frac8{n+r}\frac{\Tr(S^{+3})}{\text{tr}^3\,(S^+)}\bigg]t
\\&\quad
+\frac{n-r-2}{n+r}\bigg[\frac{n-r-2}{n+r}\frac{\Tr(S^{+2})}{\text{tr}^2(S^+)}
+\frac4{n+r}\frac{\Tr(S^{+3})}{\text{tr}^3(S^+)}
\bigg]t^2
\\&\quad
+\frac{2}{n+r}
\bigg[\!-\!\frac{n-r-2}{n+r}\frac{\Tr(S^{+2})}{\text{tr}^2\,(S^+)}
\!+2\frac{n-r-6}{n+r}\frac{\Tr(S^{+3})}{\text{tr}^3\,(S^+)}
\!+\frac{12}{n+r}\frac{\Tr(S^{+4})}{\text{tr}^4\,(S^+)}
\bigg]t^2
\\&\quad
=\;
p-\frac{(n-1)r}{n+r}
+\bigg[
-2\frac{(n-r-2)(r+1)}{(n+r)^2}
+4\frac{n-2r-5}{(n+r)^2}\frac{\Tr(S^{+2})}{\text{tr}^2(S^+)}
\\&\quad\qquad
+\frac{16}{(n+r)^2}\frac{\Tr(S^{+3})}{\text{tr}^3\,(S^+)}\bigg]t
+\bigg[\frac{(n-r-2)(n-r-4)}{(n+r)^2}\frac{\Tr(S^{+2})}{\text{tr}^2(S^+)}
\\&\quad\qquad
+8\frac{n-r-4}{(n+r)^2}\frac{\Tr(S^{+3})}{\text{tr}^3(S^+)}
+\frac{24}{(n+r)^2}\frac{\Tr(S^{+4})}{\text{tr}^4\,(S^+)}
\bigg]t^2.
\end{align*}}
Now note that $\Tr(S^{+3})\leq\Tr^{\frac12}(S^{+4})\Tr^{\frac12}(S^{+2})\leq\Tr(S^{+2})\Tr(S^{+})$ and  $\Tr(S^{+4})\leq\Tr^{\frac12}(S^{+6})\Tr^{\frac12}(S^{+2})\leq\Tr(S^{+3})\Tr(S^{+})\leq\Tr(S^{+2})\Tr^2(S^{+})$. Then since $r\leq  n-4$ and $-1\leq -\frac{\Tr(S^{+2})}{\text{tr}^2(S^+)}$ we can write
{\setlength{\mathindent}{5pt}\begin{align}&
U
\leq\;
p-\!\frac{(n-1)r}{n+r}
+\bigg[
\!-\!2\frac{(n-r-2)(r+1)}{(n+r)^2}\frac{\Tr(S^{+2})}{\text{tr}^2(S^+)}
+4\frac{n-2r-5}{(n+r)^2}\frac{\Tr(S^{+2})}{\text{tr}^2(S^+)}
\notag\\&\quad\qquad
+\frac{16}{(n+r)^2}\frac{\Tr(S^{+2})}{\text{tr}^2(S^+)}\bigg]t
+\bigg[\frac{(n-r-2)(n-r-4)}{(n+r)^2}\frac{\Tr(S^{+2})}{\text{tr}^2(S^+)}
\notag\\&\quad\qquad
+8\frac{n-r-4}{(n+r)^2}\frac{\Tr(S^{+2})}{\text{tr}^2(S^+)}
+\frac{24}{(n+r)^2}\frac{\Tr(S^{+2})}{\text{tr}^2(S^+)}
\bigg]t^2
\notag\\&\;\;
\leq\;
p-\frac{(n-1)r}{n+r}
+\bigg[
\frac{(n-r)(n-r+2)}{(n+r)^2}t^2
-2\frac{(n-r)(r-1)}{(n+r)^2}t
\bigg]\frac{\Tr(S^{+2})}{\text{tr}^2(S^+)}.
\label{eq:cov-aS+tTr-U}
\end{align}}
Now, using that $\frac{\Tr(S^{+2})}{\text{tr}^2(S^+)},\frac{\Tr(S^{+3})}{\text{tr}^3(S^+)},\frac{\Tr(S^{+4})}{\text{tr}^4(S^+)}\leq1$ we find
{\setlength{\mathindent}{5pt}\begin{align*}&
\E{\left|p+\sum_{k=1}^r \frac{n-r-2}n\frac{\psi_k}{l_k}
+\frac2n\sum_{k=1}^r\frac{\partial \psi_k}{\partial l_k}
+\frac1n\sum_{k\neq b}^r\frac{\psi_k-\psi_b}{l_k-l_b}\right|}
\\&\quad
=\E{\left|p
+\frac{n-r-2}{n+r}\sum_{k=1}^r\left[1+\frac{t}{l_k\Tr(S^+)}\right]
\right.\right.\\&\hspace{140pt}\left.\left.
+\frac2{n+r}\sum_{k=1}^r\left[1+\frac{t}{l_k^2\Tr^2(S^+)}\right]
+\frac1{n+r}\sum_{k\neq b}^r1\right|}
\\&\quad
=\left|p
+\frac{n-r-2}{n+r}(r+t)
+\frac2{n+r}(r+t)
+\frac{r(r-1)}{n+r}\right|<\infty,
\\&
\E{\left|\sum_{k=1}^r\frac{\psi_k}{l_k}\right|}
=\E{\left|\frac{n}{n+r}\sum_{k=1}^r\left[1+\frac{t}{l_k\Tr(S^+)}\right]\right|}=\frac{n}{n+r}|r+t|<\infty,
\\&
\E{\left|\sum_{k=1}^r\frac{\psi^*_k}{l_k}\right|}
=\frac{n}{n+r}\E{\left|
-r
+\bigg[
-2\frac{r+1}{n+r}
+\frac4{n+r}\frac{\Tr(S^{+2})}{\text{tr}^2(S^+)}
\bigg]t
\right.\right.\\&\hspace{140pt}\left.\left.
+\bigg[
\frac{n}{n+r}\frac{\Tr(S^{+2})}{\text{tr}^2(S^+)}
+\frac4{n+r}\frac{\Tr(S^{+3})}{\text{tr}^3(S^+)}
\bigg]t^2
\right|}
\\&\hspace{40pt}
\leq\;\frac{n}{n+r}\bigg[r
+\bigg(
2\frac{r+1}{n+r}
+\frac4{n+r}
\bigg)|t|
+\bigg(
\frac{n}{n+r}
+\frac4{n+r}
\bigg)t^2
\bigg]<\infty
\end{align*}}
and by \eqref{eq:cov-aS+tTr-U}
{\setlength{\mathindent}{5pt}\begin{align*}&
\E{\Big|U\Big|}\leq p+\frac{(n-1)r}{n+r}
+\bigg[
\frac{(n-r)(n-r+2)}{(n+r)^2}t^2
+2\frac{(n-r)(r-1)}{(n+r)^2}t
\bigg]<\infty.
\end{align*}}
Thus all the regularity conditions of Theorem \ref{thm:cov-ure} are satisfied, and we find
{\setlength{\mathindent}{5pt}\begin{align*}&
\E{\Tr\Big(\big[\hat\Sigma_t\Sigma^+-I_p\big]^2\Big)}=\E{U}
\\&\quad
\leq\;
p-\!\frac{(n-1)r}{n+r}
\!+\!\bigg[
\frac{(n-r)(n-r+2)}{(n+r)^2}t^2
\!-\!2\frac{(n-r)(r-1)}{(n+r)^2}t
\bigg]\!\E{\frac{\Tr(S^{+2})}{\text{tr}^2(S^+)}}\!,
\end{align*}}
which proves inequality \eqref{eq:cov-aS+tTr-risk}. To minimize this upper bound, notice that since $\E{\frac{\Tr(S^{+2})}{\text{tr}^2(S^+)}}\geq0$, it is enough to minimize the quadratic coefficient $\frac{(n-r)(n-r+2)}{(n+r)^2}t^2
-2\frac{(n-r)(r-1)}{(n+r)^2}t$. This is achieved precisely when $t=\frac{r-1}{n-r+2}$. When $r>1$, this makes this quadratic coefficient strictly negative, which in view of Proposition \ref{prop:cov-aS} guarantees
\[\E{\Tr\Big(\big[\hat\Sigma_\text{HF2}\Sigma^+-I_p\big]^2\Big)}<p-\frac{(n-1)r}{n+r}=\E{\Tr\Big(\big[\hat\Sigma_\text{HF1}\Sigma^+-I_p\big]^2\Big)}.\]
Thus in this case $\hat\Sigma_\text{HF2}$ dominates $\hat\Sigma_\text{HF1}$, as desired.
\end{proof}

\subsection{Proofs of Subsection \ref{subsec:pres}}

\begin{proof}[Proof of Theorem \ref{thm:pres-ure}]
Since $T$ and $S$ share the same non-zero eigenvalues, we can regard $\Psi$ as a function of $T\sim\text{W}_r(n-1,\Omega/n)$ only.
Since $r\leq n-1$ and $\Omega$ is full rank we can apply Lemma 2.1 from \citet{Dey87}. However, the proposition is given without proof and, more importantly, without the implied regularity conditions that inevitably come from using Stein's and Haff's lemmas. For completeness, we therefore derive again this result in our context.
First, we can write
{\setlength{\mathindent}{5pt}\begin{align*}&
\E{\|O_1\Psi O_1^t-H\Omega^{-1}H^t\|_F^2}
=\E{\|U\Psi U^t-\Omega^{-1}\|_F^2}
\\&\quad
=\E{\tr{U\Psi^2 U^t}-2\tr{\Omega^{-1}U\Psi U^t}}+\tr{\Omega^{-2}}
\end{align*}}
By Lemma 3 of \citet{ChetelatWells14}, this equals
{\setlength{\mathindent}{5pt}\begin{align*}&
\;\;
=\E{\sum_{k=1}^r\psi_k^2
-\!2\bigg(\frac{n-r-2}{n}\sum_{k=1}^p
\!+\frac2n\sum_{k=1}^r\frac{\partial\psi_k}{\partial l_k}
\!+\!\frac1n\sum_{k\not=b}\frac{\psi_k-\psi_b}{l_k-l_b}\bigg)}
\!\!+\!\tr{\Omega^{-2}}
\\&\;\;
=\E{
\sum_{k=1}^r\psi_k^2
-\!2\frac{n-r-2}n\sum_{k=1}^r\frac{\psi_k}{l_k}
\!-\!\frac4n\sum_{k=1}^r\frac{\partial\psi_k}{\partial l_k}
\!-\!\frac2n\sum_{k\not=b}^r\frac{\psi_k-\psi_b}{l_k-l_b}
}
\!\!+\!\tr{\Omega^{-2}}
\end{align*}}
under the regularity condition
{\setlength{\mathindent}{5pt}\begin{align*}&
\quad
\E{\bigg|
\frac{n-r-2}{n}\sum_{k=1}^p\frac{\psi_k}{l_k}
+\frac2n\sum_{k=1}^r\frac{\partial\psi_k}{\partial l_k}
+\frac1n\sum_{k\not=b}\frac{\psi_k-\psi_b}{l_k-l_b}\bigg|}<\infty.
\end{align*}}
The result follows from the fact that $\tr{\Omega^{-2}}=\tr{H^tH\Omega^{-1}H^tH\Omega^{-1}}=\tr{\Sigma^{-2}}$.
\end{proof}

\begin{proof}[Proof of Proposition \ref{prop:pres-aS}]
We have $\psi_k=a/l_k$, so
\begin{align*}&
\sum_{k=1}^r\psi_k^2
=a^2\sum_{k=1}^r\frac1{l_k^2}
=a^2\Tr(S^{+2})
\\&
-2\frac{n-r-2}n\sum_{k=1}^r\frac{\psi_k}{l_k}
=-2\frac{n-r-2}na\sum_{k=1}^r\frac1{l_k^2}
=-2\frac{n-r-2}na\Tr(S^{+2})
\\&
-\frac4n\sum_{k=1}^r\frac{\partial\psi_k}{\partial l_k}
=-\frac4na\sum_{k=1}^r-\frac1{l^2_k}
=\frac4na\Tr(S^{+2})
\\&
-\frac2n\sum_{k\not=b}^r\frac{\psi_k-\psi_b}{l_k-l_b}
=-\frac2na\sum_{k\not=b}^r\frac{l_k^{-1}-l_b^{-1}}{l_k-l_b}
=\frac2na\Tr^2(S^+)-\frac2na\Tr(S^{+2}).
\end{align*}
Summing everything, we get the URE
\begin{align*}&
U
=\frac2na\Tr^2(S^+)
+\bigg(a^2
-2\frac{n-r-3}na\bigg)\Tr(S^{+2}).
\end{align*}
Now notice that
{\setlength{\mathindent}{5pt}\begin{align*}&
\quad
\E{\bigg|\frac{n-r-2}{n}\sum_{k=1}^p\frac{\psi_k}{l_k}
+\frac2n\sum_{k=1}^r\frac{\partial\psi_k}{\partial l_k}
+\frac1n\sum_{k\not=b}\frac{\psi_k-\psi_b}{l_k-l_b}\bigg|}
\\&\qquad
=|a|\E{\bigg|\frac{n-r-3}n\Tr(S^{+2})-\frac1n\Tr^2(S^+)\bigg|}
\\&\qquad
\leq\frac{n-r-3}n\left|a\right|\E{\Tr(S^{+2})}+\frac1n|a|\E{\Tr^2(S^+)}.
\end{align*}}
Since $T\sim\text{W}_r(n-1,\Omega/n)$, by Theorem 2.4.14 (viii) from \citet{KolloVonRosen06} we have the bound
\begin{align}
\E{\Tr(S^{+2})}\leq\E{\Tr^2(S^{+})}=\E{\Tr^2(T^{-1})}<\infty
\label{eq:pres-aS-inf}
\end{align}
when $n-r-4>0$, which holds since $r\leq n-5$. Therefore, the regularity condition hold and we can apply Theorem \ref{thm:pres-ure} to conclude that
\begin{align*}&
\E{\|aS^+-\Sigma^+\|_F^2}
=\E{U}
\\&\qquad
=\frac2na\E{\Tr^2(S^+)}
+\bigg(a^2
-2\frac{n-r-3}na\bigg)\E{\Tr(S^{+2})}
\end{align*}
for any $a\in\R$. Thus, in particular, the risk of the unbiased estimator $\frac{n-r-2}nS$ must equal
$\frac{2(n-r-2)}{n^2}\E{\Tr^2(S^+)}-\frac{(n-r-2)(n-r-4)}{n^2}\E{\Tr(S^{+2})}$. When $a\leq\frac{n-r-2}n$ we can bound
{\setlength{\mathindent}{5pt}\begin{align*}&
\E{\|aS^+-\Sigma^+\|_F^2}-\E{\Big\|\frac{n-r-2}nS^+-\Sigma^+\Big\|_F^2}
\\&\qquad
=\;
\frac2n\left(a-\frac{n-r-2}{n}\right)\E{\Tr^2(S^+)}
\\&\qquad\qquad
+\bigg(a^2
-2\frac{n-r-3}na
+\frac{(n-r-2)(n-r-4)}{n^2}\bigg)\E{\Tr(S^{+2})}
\\&\qquad
=\;
\frac2n\left(a-\frac{n-r-2}{n}\right)\E{\Tr^2(S^+)}
\\&\qquad\qquad
+\bigg(a-\frac{n-r-2}{n}\bigg)\bigg(a-\frac{n-r-4}{n}\bigg)\E{\Tr(S^{+2})}
\\&\qquad
\leq\;
\bigg(a-\frac{n-r-2}{n}\bigg)\bigg(a-\frac{n-r-6}{n}\bigg)\E{\Tr(S^{+2})},
\end{align*}}
which shows inequality \ref{eq:pres-aS-risk}.
This upper bound has a minimum at $a=\frac{n-r-4}{n}$, which yields
{\setlength{\mathindent}{5pt}\begin{align*}&
\E{\big\|aS^+-\Sigma^+\big\|_F^2}-\E{\Big\|\frac{n-r-2}nS^+-\Sigma^+\Big\|_F^2}
\leq\;
-\frac4{n^2}\E{\Tr(S^{+2})}<\;0.
\end{align*}}
Thus $\frac{n-r-4}{n}S^+$ dominates $\frac{n-r-2}{n}S^+$, as desired.
Moreover, the URE of $S^+$ is $\frac2n\Tr^2(S^+)-\frac{n-2r-6}{n}\Tr(S^{+2})$ and so
{\setlength{\mathindent}{5pt}\begin{align*}&
\E{\|\frac{n-r-2}nS^+-\Sigma^+\|_F^2}-\E{\Big\|S^+-\Sigma^+\Big\|_F^2}
\\&\qquad
=\;
\E{-2\frac{r+2}{n^2}\Tr^2(S^+)
-\frac{(r+2)(r+4)}{n^2}
\Tr(S^{+2})}
\qquad\leq\;0,
\end{align*}}
so $\frac{n-r-2}nS^+$ dominates $S^+$, as claimed.
\end{proof}

\begin{proof}[Proof of Proposition \ref{prop:pres-aS-tTr}]
We have $\psi_k=a[1/l_k+t\Tr^{-1}(S)]$, so
{\setlength{\mathindent}{5pt}\begin{align*}&
\sum_{k=1}^r\psi_k^2
=\frac{(n-r-4)^2}{n^2}\sum_{k=1}^r\left[\frac1{l_k^2}+\frac{2t}{l_k\Tr(S)}+\frac{t^2}{\Tr(S)}\right]
\\&\hspace{15pt}
=\frac{(n-r-4)^2}{n^2}\Tr(S^{+2})
\!+2\frac{(n-r-4)^2}{n^2}t\frac{\Tr(S^+)}{\Tr(S)}
\!+\frac{(n-r-4)^2r}{n^2}t^2\frac1{\Tr^2(S)}
\\&
-2\frac{n-r-2}n\sum_{k=1}^r\frac{\psi_k}{l_k}
=-2\frac{(n-r-2)(n-r-4)}{n^2}\sum_{k=1}^r\left[\frac1{l_k^2}+\frac{t}{l_k\Tr(S)}\right]
\\&\hspace{15pt}
=-2\frac{(n-r-2)(n-r-4)}{n^2}\,\Tr(S^{+2})
-2\frac{(n-r-2)(n-r-4)}{n^2}t\frac{\Tr(S^+)}{\Tr(S)}
\\&
-\frac4n\sum_{k=1}^r\frac{\partial\psi_k}{\partial l_k}
=-4\frac{n-r-4}{n^2}\sum_{k=1}^r\left[-\frac1{l^2_k}-\frac{t}{\Tr^2(S)}\right]
\\&\hspace{15pt}
=4\frac{n-r-4}{n^2}\,\Tr(S^{+2})+4\frac{(n-r-4)r}{n^2}t\frac1{\Tr^2(S)}
\\&
-\frac2n\sum_{k\not=b}^r\frac{\psi_k-\psi_b}{l_k-l_b}
=-2\frac{n-r-4}{n^2}\sum_{k\not=b}^r\frac{l_k^{-1}-l_b^{-1}}{l_k-l_b}
\\&\hspace{15pt}
=2\frac{n-r-4}{n^2}\Tr^2(S^+)-2\frac{n-r-4}{n^2}\Tr(S^{+2}).
\end{align*}}
Summing everything, we get the URE
{\setlength{\mathindent}{5pt}\begin{align*}&
U
=
2\frac{n-r-4}{n^2}\Tr^2(S^+)-\frac{(n-r-4)(n-r-2)}{n^2}\Tr(S^{+2})
\\&\qquad
+4\frac{n-r-4}{n^2}\bigg[
r\frac1{\Tr^2(S)}
-\frac{\Tr(S^+)}{\Tr(S)}
\bigg]t
+\frac{(n-r-4)^2r}{n^2}t^2\frac1{\Tr^2(S)}.
\end{align*}}
Now note, using $\Tr^{-1}(S)\leq\Tr(S^+)/r^2$ and $\Tr(S^{+2})\leq\Tr^2(S^+)$ that
{\setlength{\mathindent}{5pt}\begin{align*}&
\E{\bigg|\frac{n-r-2}{n}\sum_{k=1}^p\frac{\psi_k}{l_k}
+\frac2n\sum_{k=1}^r\frac{\partial\psi_k}{\partial l_k}
+\frac1n\sum_{k\not=b}\frac{\psi_k-\psi_b}{l_k-l_b}\bigg|}
\\&\qquad
=\frac{(n-r-3)(n-r-4)}{n^2}
\E{\Tr(S^{+2})}
+\frac{n-r-4}{n^2}\E{\Tr^2(S^+)}
\\&\hspace{40pt}
+\frac{(n-r-2)(n-r-4)}{n^2}t\E{\frac{\Tr(S^+)}{\Tr(S)}}
-2\frac{(n-r-4)r}{n^2}t\E{\frac1{\Tr^2(S)}}
\\&\qquad
\leq\Big(\frac{(n-r-1)(n-r-4)}{n^2}
+\frac{(n-r-2)(n-r-4)}{r^2n^2}|t|
\\&\hspace{140pt}
+2\frac{n-r-4}{r^3n^2}|t|\Big)
\E{\Tr^2(S^+)}
\qquad<\infty,
\end{align*}}
since $\E{\Tr^2(S^+)}<\infty$ by equation \eqref{eq:pres-aS-inf}.
Therefore, we can apply Theorem \ref{thm:pres-ure} to obtain
{\setlength{\mathindent}{5pt}\begin{align*}&
\E{\Big\|\hat\Sigma^+_t-\Sigma^+\Big\|_F^2}=\E{U}
\\&\qquad
=\;2\frac{n-r-4}{n^2}\E{\Tr^2(S^+)}
-\frac{(n-r-4)(n-r-2)}{n^2}\E{\Tr(S^{+2})}
\\&\hspace{40pt}
+4\frac{n-r-4}{n^2}t\E{
r\frac1{\Tr^2(S)}
-\frac{\Tr(S^+)}{\Tr(S)}
}
+\frac{(n-r-4)^2r}{n^2}t^2\E{\frac1{\Tr^2(S)}}
\end{align*}}
for all $t\in\R$. Using that $n-r-4>0$ and $r^2\Tr^{-1}(S)\leq\Tr(S^+)$ again, we can bound the difference in risk as
{\setlength{\mathindent}{5pt}\begin{align*}&
\E{\Big\|\hat\Sigma^+_t-\Sigma^+\Big\|_F^2}-\E{\Big\|\hat\Sigma^+_\text{EM1}-\Sigma^+\Big\|_F^2}
\\&\qquad
\leq\;
\frac{(n-r-4)r}{n^2}\bigg[
(n-r-4)t^2
-4(r-1)t
\bigg]
\E{\frac1{\Tr^{2}(S)}}
\end{align*}}
which proves inequality \eqref{eq:pres-aS-tTr-risk}.
There is a minimum in $t$ since $n-r-4>0$, which is $t=2\frac{r-1}{n-r-4}$. In this case the quadratic coefficient and thus the difference in risk is strictly negative, so the corresponding estimator $\hat\Sigma_\text{EM2} =\frac{n-r-4}n\left[S^++2\frac{r-1}{n-r-4}\,\Tr^{-1}(S)\right]$ dominates $\hat\Sigma^+_\text{EM1}$, as desired.
\end{proof}

\subsection{Proofs of Subsection \ref{subsec:disc}}

\begin{proof}[Proof of Theorem \ref{thm:disc-ure}]
Since $T$ and $S$ share the same non-zero eigenvalues, we can regard $\Psi$ as a function of $T\sim\text{W}_r(n-1,\Omega/n)$ only.
Moreover, $\bar X = \mu+H\bar Y$. Using that $O_1O_1^t=HH^t$ almost surely, we find
{\setlength{\mathindent}{10pt}\begin{align*}
&\E{\Big\|\hat{\Sigma}^+\bar X-\Sigma^+\mu\Big\|_2^2}
=
\E{\Big\|O_1O_1^tO_1\Psi O_1^tO_1O_1^t[\mu+H\bar Y] -H\Omega^{-1}H^t\mu\Big\|_2^2}
\\&\qquad
=
\E{\Big\|U\Psi U^t[H^t\mu+\bar Y] -\Omega^{-1}H^t\mu\Big\|_2^2}
\end{align*}}
Define $G=H^t\mu+\bar Y\sim\text{N}_r(H^t\mu,\Omega/n)$ and notice it is independent of $U\Psi U^t$ since $\bar X$ and $S$ are independent. Then
{\setlength{\mathindent}{10pt}\begin{align*}
&\qquad
=\E{\Big\|U\Psi U^tW -\Omega^{-1}H^t\mu\Big\|_2^2}
\\&\qquad
=2\E{(G-H^t\mu)^t\Omega^{-1}U\Psi U^tG}-2\E{\tr{\Omega^{-1}U\Psi U^tGG^t}}
\\&\qquad\qquad
+\E{G^tU\Psi^2U^tG}-\E{(G-H^t\mu)^t\Omega^{-2}(G+H^t\mu)}.
\end{align*}}
The first term can be handled as follows. By independence of $G$ and $U\Psi U^t$, and Stein's lemma \citep[Lemma A.1]{FourdrinierStrawderman03}, we get
{\setlength{\mathindent}{10pt}\begin{align*}&
2\E{(G-H^t\mu)^t\Omega^{-1}U\Psi U^tG}
=\frac2n\text{E}_G\left[(G-H^t\mu)^t\left[\frac{\Omega}n\right]^{-1}\text{E}_T\left[U\Psi U^t\right]G\right]
\\&\qquad
=\frac2n\text{E}_G\left[\nabla_GG'\text{E}_T\left[U\Psi U^t\right]\right]
=\frac2n\text{Etr}\left[\Psi\right]
\end{align*}}
under the condition
{\setlength{\mathindent}{10pt}\begin{align*}&
\text{E}_G\left[\Big|\nabla_GG'\text{E}_T\left[U\Psi U^t\right]\Big|\right]
=\text{E}_G\left[\Big|\tr{\Psi}\Big|\right]
=\E{\left|\sum_{k=1}^r\psi_k\right|}<\infty.
\end{align*}}
For the second term, we will make use of the fact that
\begin{align}
\text{E}_T\left[\Omega^{-1}U\Psi U^t\right]=\text{E}_T\left[U\Psi^* U^t\right],
\label{eq:thm:disc-ure-haff}
\end{align}
where $\Psi^*$ is defined as the statement. This is the result of a  non-singular analogue of Theorem 2.2 from \citet{Konno09}, or alternatively of a matrix analogue of Lemma 3 from \citet{ChetelatWells14}. By appropriate modifications to the latter result and the underlying Lemma 3 from \cite{ChetelatWells12} on which it depends, it can be seen that sufficient conditions for equation \ref{eq:thm:disc-ure-haff} to hold are
\begin{align*}
\text{E}_T\left[\big|U\Psi^* U^t\big|_{ij}\right]<\infty
\hspace{80pt}\forall 1\leq i,j\leq r.
\end{align*}
A sufficient condition for this to happen is
\begin{align*}
\max_{1\leq i,j\leq r}\text{E}_T\left[\big|U\Psi^* U^t\big|_{ij}\right]
&\leq\;
\E{\sum_{k=1}^r\Big|\psi^*_k\Big|}<\infty.
\end{align*}
Then, using the independence of $G$ and $T$, we can conclude
{\setlength{\mathindent}{10pt}\begin{align*}&
-2\E{\tr{\Omega^{-1}U\Psi U^tGG^t}}
=-2\tr{\text{E}_T\left[\Omega^{-1}U\Psi U^t\right]\text{E}_G\left[GG^t\right]}
\\&\qquad
=-2\tr{\text{E}_T\left[U\Psi^* U^t\right]\text{E}_G\left[GG^t\right]}
=-2\E{G^tU\Psi^* U^tG}.
\end{align*}}
Thus
{\setlength{\mathindent}{10pt}\begin{align*}
\E{\Big\|\hat{\Sigma}^+\bar X-\Sigma^+\mu\Big\|_F^2}
&
=\frac2n\E{\tr{\Psi}}
-2\E{G^tU\Psi^* U^tG}
+\E{G^tU\Psi^2U^tG}
\\&\qquad
-\E{(G-H^t\mu)^t\Omega^{-2}(G+H^t\mu)}.
\end{align*}}
But $U^tG=O_1^tH[H^t\mu+\bar Y]=O_1^t\bar X$ and $(G-H^t\mu)^t\Omega^{-2}(G+H^t\mu)=(G-H^t\mu)^tH^t\Sigma^{+2}H(G+H^t\mu)=(\bar X-\mu)^t\Sigma^{+2}(\bar X+\mu)$. Hence
{\setlength{\mathindent}{10pt}\begin{align*}
\E{\Big\|\hat{\Sigma}^+\bar X-\Sigma^+\mu\Big\|_2^2}
&=\E{\frac2n\text{tr}\,\hat{\Sigma}^++\bar X^tO_1(\Psi^2-2\Psi^*)O_1^t\bar X}
\\&\qquad-\E{(\bar X-\mu)^t\Sigma^{+2}(\bar X+\mu)}.
\end{align*}}
This proves the result.
\end{proof}

\begin{proof}[Proof of Proposition \ref{prop:disc-aS}]
We have $\psi_k=a/l_k$, so
{\setlength{\mathindent}{10pt}\begin{align*}
\psi^*_k
&=\frac{n-r-2}n\frac{\psi_k}{l_k}
+\frac2n\frac{\partial\psi_k}{\partial l_k}
+\frac1n\sum_{b\not=k}^r\frac{\psi_k-\psi_b}{l_k-l_b}
\\&
=\frac{n-r-2}n\frac1{l^2_k}a
-\frac2n\frac1{l^2_k}a
-\frac1n\frac{\Tr(S^+)}{l_k}a
+\frac1n\frac1{l^2_k}a
\\&
=\frac{n-r-3}n\frac1{l^2_k}a
-\frac1n\frac{\Tr(S^+)}{l_k}a.
\end{align*}}
We can bound
\begin{align*}
\E{\left|\sum_{k=1}^r\psi_k\right|}
=|a|\E{\Tr(S^+)}
\leq\;|a|\E{\Tr^2(S^+)}^\frac12,
\end{align*}
\begin{align*}
\E{\sum_{k=1}^r\Big|\psi^*_k\Big|}
\leq
\frac{n-r-3}n|a|\E{\Tr(S^{+2})}
+\frac1n|a|\E{\Tr^2(S^+)},
\end{align*}
so by inequality \eqref{eq:pres-aS-inf} and the fact that $n-r-4>0$ these two expressions are finite. Therefore, we can apply the results of Theorem \ref{thm:disc-ure} to obtain
{\setlength{\mathindent}{10pt}\begin{align*}&
\E{\Big\|aS^+\bar X-\Sigma^+\mu\Big\|_2^2}
=\frac2na\E{\Tr(S^+)}
\\&\hspace{40pt}
+\E{\sum_{k=1}^r
\bigg(
\frac{a}{l_k^2}-2\frac{n-r-3}n\frac1{l^2_k}
+\frac2n\frac{\Tr(S^+)}{l_k}
\bigg)a\left(O_1^t\bar X\bar X^tO_1\right)_{kk}}
\\&\hspace{40pt}
-\E{\vphantom{\bigg\vert}(\bar X-\mu)^t\Sigma^{+2}(\bar X+\mu)}
\\&\quad
=\frac2na\E{\Tr(S^+)}
+\bigg(
a^2
-2\frac{n-r-3}na
\bigg)\,\E{\bar X^tS^{+2}\bar X}
\\&\hspace{40pt}
+\frac2na\,\E{\vphantom{\Big\vert}\Tr(S^+)\bar X^tS^+\bar X}
-\E{\vphantom{\Big\vert}(\bar X-\mu)^t\Sigma^{+2}(\bar X+\mu)}
\end{align*}}
for any $a\in\R$. Therefore, for $a\leq\frac{n-r-2}n$ we can bound the difference in risk by
{\setlength{\mathindent}{10pt}\begin{align*}&
\E{\Big\|aS^+\bar X-\Sigma^+\mu\Big\|_2^2}
-\E{\Big\|\frac{n-r-2}{n}S^+\bar X-\Sigma^+\mu\Big\|_2^2}
\\&\quad
=\frac2n\left(a-\frac{n-r-2}n\right)\E{\Tr(S^+)}
\\&\hspace{40pt}
+\bigg(a^2
-2\frac{n-r-3}na
+\frac{(n-r-2)(n-r-4)}{n^2}
\bigg)\,\E{\bar X^tS^{+2}\bar X}
\\&\hspace{40pt}
+\frac2n\left(a-\frac{n-r-2}n\right)
\,\E{\vphantom{\Big\vert}\Tr(S^+)\bar X^tS^+\bar X}
\\&\quad
\leq\;\bigg(a-\frac{n-r-2}{n}\bigg)\bigg(a-\frac{n-r-4}{n}\bigg)\,\E{\bar X^tS^{+2}\bar X},
\end{align*}}
which proves inequality \eqref{eq:disc-aS-risk}. The quadratic coefficient is minimized at $a=\frac{n-r-3}n$, at which point we have
{\setlength{\mathindent}{10pt}\begin{align*}&
\E{\Big\|\frac{n-r-3}{n}S^+\bar X-\Sigma^+\mu\Big\|_2^2}
-\E{\Big\|\frac{n-r-2}{n}S^+\bar X-\Sigma^+\mu\Big\|_2^2}
\\&\hspace{40pt}
\leq\;-\frac1{n^2}\E{\bar X^tS^{+2}\bar X}
\qquad<\;0.
\end{align*}}
Thus $\frac{n-r-3}{n}S^+\bar X$ dominates $\frac{n-r-2}{n}S^+\bar X$, as desired. Moreover,
{\setlength{\mathindent}{5pt}\begin{align*}&
\E{\|\frac{n-r-2}nS^+\bar X-\Sigma^+\bar X\|_2^2}-\E{\Big\|S^+\bar X-\Sigma^+\bar X\Big\|_2^2}
\\&\quad
=-2\frac{r+2}{n^2}\E{\Tr(S^+)}
-\frac{(r+2)(r+4)}{n^2}\,\E{\bar X^tS^{+2}\bar X}
\\&\hspace{40pt}
-2\frac{r+2}{n^2}
\,\E{\vphantom{\Big\vert}\Tr(S^+)\bar X^tS^+\bar X}
\qquad<\;0,
\end{align*}}
so $\frac{n-r-2}nS^+$ dominates $S^+$, as claimed.
\end{proof}

\begin{proof}[Proof of Proposition \ref{prop:disc-aS-tTr}]
We will apply \ref{thm:pres-ure}, and we have here $\psi_k=\frac{n-r-3}n[1/l_k+t\Tr^{-1}(S)]$ for $1\leq k\leq r$, so
{\setlength{\mathindent}{10pt}\begin{align*}
\frac{n-r-2}n\frac{\psi_k}{l_k}
&=\frac{(n-r-2)(n-r-3)}{n^2}\left[\frac1{l_k^2}+\frac{t}{l_k\Tr(S)}\right],
\\
\frac2n\sum_{k=1}^r\frac{\partial\psi_k}{\partial l_k}
&=2\frac{n-r-3}{n^2}\left[-\frac1{l^2_k}-\frac{t}{\Tr^2(S)}\right],
\\
\frac1n\sum_{b\not=k}^r\frac{\psi_k-\psi_b}{l_k-l_b}
&=\frac{n-r-3}{n^2}\sum_{b\not=k}^r\frac{l_k^{-1}-l_b^{-1}}{l_k-l_b}
=\frac{n-r-3}{n^2}\left[\frac1{l_k^2}-\frac{\Tr(S^+)}{l_k}\right].
\end{align*}}
Therefore,
{\setlength{\mathindent}{20pt}\begin{align*}
\psi^*_k
&=\frac{n-r-2}n\frac{\psi_k}{l_k}
+\frac2n\frac{\partial\psi_k}{\partial l_k}
+\frac1n\sum_{b\not=k}^r\frac{\psi_k-\psi_b}{l_k-l_b}
\\&
=\frac{(n-r-3)^2}{n^2}\frac1{l_k^2}
+\frac{(n-r-2)(n-r-3)}{n^2}t\,\frac{\Tr^{-1}(S)}{l_k}
\\&\qquad
-2\frac{n-r-3}{n^2}t\,\Tr^{-2}(S)
-\frac{n-r-3}{n^2}\frac{\Tr(S^+)}{l_k}
\end{align*}}
We can bound
{\setlength{\mathindent}{10pt}\begin{align*}
\E{\left|\sum_{k=1}^r\psi_k\right|}
\leq\;
\frac{n-r-3}n\E{\Tr(S^+)}+\frac{(n-r-3)r}n|t|\E{\Tr^{-1}(S)},
\end{align*}}
{\setlength{\mathindent}{10pt}\begin{align*}
\E{\sum_{k=1}^r\Big|\psi^*_k\Big|}
&\leq\;
\frac{(n-r-3)^2}{n^2}\E{\Tr(S^{+2})}
+\frac{n-r-3}{n^2}\E{\Tr^2(S^+)}
\\&\qquad
+\frac{(n-r)(n-r-3)}{n^2}|t|\,\E{\Tr^{-2}(S)},
\end{align*}}
so by $\Tr^{-1}\leq\Tr(S^+)/r^2$, inequality \eqref{eq:pres-aS-inf} and the fact that $n-r-4>0$ these two expressions are finite. Therefore, we can apply the results of Theorem \ref{thm:disc-ure} to obtain
{\setlength{\mathindent}{10pt}\begin{align*}&
\E{\Big\|\hat{\eta}_t-\eta\Big\|_2^2}
=2\frac{n-r-3}{n^2}\E{\Tr(S^+)}
+2\frac{(n-r-3)r}{n^2}t\E{\Tr^{-1}(S)}
\\&\hspace{20pt}
+\E{\sum_{k=1}^r\bigg(
\frac{(n-r-3)^2}{n^2}\frac1{l_k^2}
+2\frac{(n-r-3)^2}{n^2}t\frac{\Tr^{-1}(S)}{l_k}
\right.\\&\hspace{60pt}\left.
+\frac{(n-r-3)^2}{n^2}t^2\Tr^{-2}(S)
-2\frac{(n-r-3)^2}{n^2}\frac1{l_k^2}
\right.\\&\hspace{60pt}\left.
-2\frac{(n-r-2)(n-r-3)}{n^2}t\,\frac{\Tr^{-1}(S)}{l_k}
+4\frac{n-r-3}{n^2}t\,\Tr^{-2}(S)
\right.\\&\hspace{60pt}\left.
+2\frac{n-r-3}{n^2}\frac{\Tr(S^+)}{l_k}
\bigg)\left(O_1^t\bar X\bar X^tO_1\right)_{kk}}
\\&\hspace{20pt}
-\E{\vphantom{\bigg\vert}(\bar X-\mu)^t\Sigma^{+2}(\bar X+\mu)}
\\&\quad
=2\frac{n-r-3}{n^2}\E{\Tr(S^+)}
+2\frac{(n-r-3)r}{n^2}t\E{\Tr^{-1}(S)}
\\&\hspace{20pt}
+\E{\sum_{k=1}^r
\bigg(-\frac{(n-r-3)^2}{n^2}\frac1{l_k^2}
+2\frac{n-r-3}{n^2}\frac{\Tr(S^+)}{l_k}
\right.\\&\hspace{60pt}\left.
-2\frac{n-r-3}{n^2}t\frac{\Tr^{-1}(S)}{l_k}
+4\frac{n-r-3}{n^2}t\,\Tr^{-2}(S)
\right.\\&\hspace{60pt}\left.
+\frac{(n-r-3)^2}{n^2}t^2\Tr^{-2}(S)
\bigg)\left(O_1^t\bar X\bar X^tO_1\right)_{kk}}
\\&\hspace{20pt}
-\E{\vphantom{\bigg\vert}(\bar X-\mu)^t\Sigma^{+2}(\bar X+\mu)}
\\&\quad
=2\frac{n-r-3}{n^2}\E{\Tr(S^+)}
-\frac{(n-r-3)^2}{n^2}\E{\bar X^tS^{+2}\bar X}
\\&\hspace{20pt}
+2\frac{n-r-3}{n^2}\E{\Tr(S^+)\bar X^tS^+\bar X}
+\bigg(2\frac{(n-r-3)r}{n^2}\E{\Tr^{-1}(S)}
\\&\hspace{20pt}
-2\frac{n-r-3}{n^2}\E{\frac{\bar X^tS^+\bar X}{\Tr(S)}}
+4\frac{n-r-3}{n^2}\E{\frac{\bar X^t\bar X}{\Tr^2(S)}}\bigg)t
\\&\hspace{20pt}
+\frac{(n-r-3)^2}{n^2}t^2\E{\frac{\bar X^t\bar X}{\Tr^2(S)}}
-\E{\vphantom{\bigg\vert}(\bar X-\mu)^t\Sigma^{+2}(\bar X+\mu)}
\end{align*}}
for any $t\in\R$. Therefore, the difference in risk can be written
{\setlength{\mathindent}{10pt}\begin{align*}&
\E{\Big\|\hat{\eta}_t-\eta\Big\|_2^2}
-\E{\Big\|\frac{n-r-3}nS^+\bar X-\eta\Big\|_2^2}
\\&\quad
=\bigg(2\frac{(n-r-3)r}{n^2}\E{\Tr^{-1}(S)}
-2\frac{n-r-3}{n^2}\E{\frac{\bar X^tS^+\bar X}{\Tr(S)}}
\\&\hspace{40pt}
+4\frac{n-r-3}{n^2}\E{\frac{\bar X^t\bar X}{\Tr^2(S)}}\bigg)t
+\frac{(n-r-3)^2}{n^2}t^2\E{\frac{\bar X^t\bar X}{\Tr^2(S)}}.
\end{align*}}
But $\Tr(\bar X\bar X^t)=\Tr(SS^+\bar X\bar X^t)\leq\Tr^\frac12(S^2)\Tr^\frac12([S^+\bar X\bar X^t]^2)\leq\Tr(S)\Tr(S^+\bar X\bar X^t)$, so we can bound
{\setlength{\mathindent}{10pt}\begin{align*}
&\quad
\leq\; 2\frac{(n-r-3)r}{n^2}t\E{\Tr^{-1}(S)}
+2\frac{n-r-3}{n^2}t\E{\frac{\bar X^t\bar X}{\Tr^2(S)}}
\\&\qquad
+\frac{(n-r-3)^2}{n^2}t^2\E{\frac{\bar X^t\bar X}{\Tr^2(S)}}.
\end{align*}}
Next, write the reduced singular value decomposition of $X$ as $\sqrt{n}V_1L^{1/2}O_1$ with $V_1$ $n\times r$ semi-orthogonal, $V^t_1V_1=I_r$. Then
\begin{align*}
\bar X^t\bar X
&=\Tr\Big(X^t\frac{1_n1^t_n}{n^2}X\Big)
=\Tr\Big(LV_1^t\frac{1_n1^t_n}{n}V_1\Big)
\\
&\leq\Tr(L)\sigma_\text{max}\bigg(V_1^t\frac{1_n1^t_n}{n}V_1\bigg)
\leq\Tr(S)\sigma_\text{max}\bigg(\frac{1_n1^t_n}{n}\bigg)
=\Tr(S).
\end{align*}
Therefore, we can bound by
{\setlength{\mathindent}{10pt}\begin{align*}
&\quad
\leq\; \frac{(n-r-3)}{n^2}\bigg[2(r+1)t
+(n-r-3)t^2\bigg]\E{\frac1{\Tr(S)}},
\end{align*}}
which proves \eqref{eq:disc-aS-tTr-risk}. Since $n-r-3>0$, the quadratic coefficient has a minimum, at $t=-\frac{r+1}{n-r-3}$. In this case we have
{\setlength{\mathindent}{10pt}\begin{align*}&
\E{\Big\|\frac{n-r-3}{n}\!\left[S^+\!\!-\!\!\frac{(r+1)\Tr^{-1}(S)}{n-r-3}\right]\!\!\bar X-\eta\Big\|_2^2}
-\E{\Big\|\frac{n-r-3}nS^+\bar X-\eta\Big\|_2^2}
\\&\quad
\leq -\frac{(r+1)^2}{n^2}\E{\frac1{\Tr(S)}}\qquad<\;0.
\end{align*}}
Thus $\hat\eta_\text{TK2}=\frac{n-r-3}{n}\left[S^+-\frac{r+1}{n-r-3}\Tr^{-1}(S)\right]$ dominates $\hat\eta_\text{TK1}$, as desired.
\end{proof}

\bibliographystyle{elsarticle-num-author}
\bibliography{covure}

\end{document}